\documentclass[reqno]{amsart}

\usepackage{amsmath}
\usepackage{a4wide}
\usepackage{stmaryrd,mathrsfs,amsthm,amssymb,mathtools,esint,color}

\newtheorem{theorem}{Theorem}[section]
\newtheorem{lemma}[theorem]{Lemma}
\newtheorem{proposition}[theorem]{Proposition}

\theoremstyle{definition}
\newtheorem{definition}[theorem]{Definition}
\newtheorem{remark}[theorem]{Remark}

\numberwithin{equation}{section}

\title{Droplet Minimizers of an Isoperimetric Problem with long-range interactions}

\author{Marco Cicalese}
\address[Marco Cicalese]{Dipartimento di Matematica e Applicazioni ``R. Caccioppoli'', Universit\`a di Napoli Federico II, Via Cintia, 80126 Napoli, Italy}
\email{cicalese@unina.it}

\author{Emanuele Spadaro}
\address[Emanuele Spadaro]{Max-Planck-Institut f\"ur Mathematik in den Naturwissenschaften, Inselstrasse 22-24, 04103 Leipzig, Germany}
\email{spadaro@mis.mpg.de}

\date{15 September 2011}

\newcommand\R{{\mathbb R}}
\newcommand\N{{\mathbb N}}
\newcommand\T{{\mathbb T}}
\newcommand{\eps}{{\varepsilon}}
\newcommand\sdif{{\vartriangle}}

\newcommand{\de}{\partial}
\newcommand{\ph}{\varphi}
\newcommand{\cC}{\mathcal{C}}
\newcommand{\cH}{\mathcal{H}}
\newcommand{\s}{\mathbb{S}}
\newcommand{\Per}{\textup{Per}}
\newcommand{\NL}{\textup{NL}}
\newcommand{\Id}{\textup{Id}}
\newcommand{\Om}{\Omega}
\newcommand{\dist}{\textup{dist}}

\newcommand{\asym}{\alpha}

\begin{document}
\begin{abstract}
We give a detailed description of the geometry of single droplet patterns in a nonlocal isoperimetric problem.
In particular we focus on the sharp interface limit of the Ohta--Kawasaki free energy for diblock copolymers, regarded as a paradigm for those energies modeling physical systems characterized by a competition between short and a long-range interactions.
Exploiting fine properties of the regularity theory for minimal surfaces, we extend previous partial results in different directions and give robust tools for the geometric analysis of more complex patterns.
\end{abstract}

\maketitle

%%%%%%%%%%%%%%% Introduction %%%%%%%%%%%%%%%%%%%%%%%%%%%%%%%%%%%%%%%
%%%%%%%%%%%%%%%%%%%%%%%%%%%%%%%%%%%%%%%%%%%%%%%%%%%%%%%%%%%%%%%%%%%%
%%%%%%%%%%%%%%%%%%%%%%%%%%%%%%%%%%%%%%%%%%%%%%%%%%%%%%%%%%%%%%%%%%%%
%%%%%%%%%%%%%%%%%%%%%%%%%%%%%%%%%%%%%%%%%%%%%%%%%%%%%%%%%%%%%%%%%%%%
%%%%%%%%%%%%%%%%%%%%%%%%%%%%%%%%%%%%%%%%%%%%%%%%%%%%%%%%%%%%%%%%%%%%

\section{Introduction}
In several physical systems competing short-range attractive and long-range repulsive interactions often lead to the formation of mesoscopic scale patterns. Roughly speaking, the short-range interactions favor phase-separation on a microscopic scale, while the long-range ones frustrate such an ordering on the scale of the whole sample. When these systems can be described in terms of a free energy, such a phenomenon is usually referred to as an \textit{energy-driven} pattern formation. Examples of energy-driven patterns are ubiquitous in physics: among the others we recall ferromagnetic and polymeric systems, type-I superconductor films and Langmuir layers.  Even if these systems are driven by different physical laws, they exhibit remarkable similarities in the overall geometry of the formed patterns (see \cite{KohICM} and \cite{SeAn}). 

Our principal interest is the description of the geometry of patterns. For this reason, we leave for further studies the detailed analysis of more realistic systems and we focus here on an energy model which encodes only the main features of pattern-formation. More specifically, in what follows we are interested in the minimization of the following energy functional: 
\begin{equation*}%\label{e.OK0}
F_{\gamma,m}(u):=\int_{\R^n} |D u|+\gamma\int_\Omega\int_\Omega G(x,y)\,\big(u(x)-m\big)\big(u(y)-m\big)\,dx\,dy,
\end{equation*}
where $u$ is the order parameter of a two-phases system confined in $\Omega\subset\R^n$, and $\gamma$ and $m$ are two nonnegative numerical parameters. The two terms in the energy mimic attractive short-range and repulsive long-range energies between the phases. More precisely
the first term is local, it favors minimal interface area and drives the system toward a partition into few pure phases while the second term, involving a Coulomb-like kernel $G$, is non-local and favor a fine mixing of the phases. A detailed description of the energy is given in \S~\ref{s.notation}. The competition between these two terms is expected to induce the formation of highly regular mesoscopic patterns, whose geometry strongly depends on the choice of the parameters $\gamma$ and $m$ (e.g.~spherical spots, cylinders, gyroids and lamellae).

\subsection{The Ohta--Kawasaki functional for diblock copolymers}
The model we consider arises as a simplification of a Ginzburg-Landau functional proposed by Ohta and Kawasaki in their pioneering paper \cite{OK} as a possible description of a diblock copolymers' (DBC) system.
Even though it is questionable whether such an energy actually describes DBCs (see Choksi and Ren \cite{ChRe03}, Muratov \cite{Mu} and Niethammer and Oshita \cite{NiOs}), nevertheless it is a first, and mathematically non-trivial, attempt to capture some of the main features of these systems.
For such a reason it deserved over the last twenty years great attention from both the mathematical and the physical community (see e.g.~\cite{BaFr,deG,Mat,OK,Stil}).
Under several simplification, the Otha-Kawasaki functional can be written in the following form:
\begin{equation}\label{e.OK}
{\mathcal E}_{\eps,\sigma}(u)=\int_\Omega\left(\eps^2|\nabla u|^2+W(u)\right)\, dx+\sigma\int_\Omega\int_\Omega
G(x,y)\,(u(x)-m)\,(u(y)-m)\, dx\, dy,
\end{equation}
where the order parameter $u$ stems for the volume fraction of one block copolymer and $W$ is a standard double-well potential.
Here $m$ is the average of $u$ over the whole sample, namely $m=\fint_\Omega u$, and the kernel $G$ solves in $\Omega$ the Neumann problem for 
\begin{equation}\label{e.introG}
-\Delta G(x,\cdot)=\delta_x -\frac{1}{|\Omega|}, \qquad \int_{\Omega}G(x,y)\,dy=0.
\end{equation}
As in the classical Ginzburg--Landau energy functional, the first contribution to the energy forces a phase separation through the competition between the gradient and the non-convex potential. On the other hand, depending on the strength of the coupling constant $\sigma$, the long-range contribution favors a uniform distribution of the order parameter. This term has an entropic origin in the case of DBC (see \cite{BaFr,deG,Mat,OK,Stil}), but it can also be considered as the energy due to an electrostatic interaction between charged bodies if the order parameter is meant to be a density of charges (\cite{ChKh,EmKi}).

It is well-known from the results by Modica and Mortola \cite{MoMo, Mo} that, when $\eps\ll 1$, the Ginzburg-Landau energy can be approximated in the sense of $\Gamma$-convergence by a sharp interface energy of the form $\eps\int_\Omega|D u|$, with $u$ being a function of bounded variations taking the two values $0,1$ (these values identify the pure phases of the system as the sets $\{x:\,u(x)=0\}$ and $\{x:\,u(x)=1\}$) and $|D u|$ denoting the total variation of the measure $Du$. Formally, this fact gives the link between the Ohta--Kawasaki energy and the functional $F_{\gamma,m}$ (for $\gamma=\sigma/\eps$). It is worth pointing out that there exists no rigorous derivation of $F_{\gamma,m}$ from $\mathcal{E}_{\eps,\sigma}$ in the sense of $\Gamma$-convergence. Indeed, the presence of possible multiple scales, (e.g., the one of the phase separation and that of the pattern formation) could force the $\Gamma$-limit to be defined on more complex spaces of Young measures as it happens in the one dimensional case addressed by Alberti and M\"uller \cite{AlMu}.

\subsection{Single droplet minimizers}
Physical experiments (see, e.g., \cite{Kh-al} and \cite{SeAn}) suggest that, for some regimes of the parameters $\gamma$ and $m$, droplets equilibrium configurations are expected.
The main open issues in this regards are: (1) the rigorous justification of the observed lattice-type patterns (for example, the Abrikosov lattice in two dimensions) and (2) the description of the geometry of the droplets.

In this paper we contribute to the second question.
In particular, we investigate a regime of $\gamma$ and $m$ leading to the formation of a \textit{single} droplet minimizer, as a first step towards the analysis of multiple droplets patterns. Roughly speaking, a single droplet minimizer can be described as a connected region of one phase surrounded by the other one.
For this to happen, the competition between the two terms of the energy has to be unbalanced, with the confining term stronger than the nonlocal one. Moreover, it is also necessary to assume a contribution to the energy coming from the interaction with the boundary of $\Omega$. This is the reason why in the functional $F_{\gamma, m}$ the total variation of $D u$ is taken in the whole space. If this were not the case, the optimal resulting shape would in fact be an almost half ball located in a point of smallest mean curvature of $\de \Omega$. An analysis of this event, though interesting in its own, is not pursued here.

In order to identify the correct regime leading to a single droplet minimizer, we show here the different contributions to the energy of a single ball. As shown in \eqref{e.exp ball}, given a ball $B_{r_m}(p)\subset\Omega$ of radius $r_m$ centered at $p$ having average mass $m$, i.e.~$m|\Omega|=\,\omega_n\,r_m^n$ (here $|\Omega|$ stands for the $n$-dimensional volume of $\Omega$), it holds
\begin{align*}
F_{\gamma,m}(\chi_{B_{r_m}(p)})=
\begin{cases}
2\,\pi\,r_m+\gamma\,\left(\frac{\pi}{2}\,r_m^4\,\log r+\left(\pi^2\,g_{r_m}(p)-\frac{3\,\pi}{8}\right)\,r_m^4\right), &\text{if }\,n=2,\\
n\,\omega_n\,r_m^{n-1}+\gamma\,\left(\frac{2\,\omega_n}{4-n^2}\,r_m^{n+2}+\omega_n^2\,g_{r_m}(p)\, r_m^{2n}\right),&\text{if }\,n\geq3,
\end{cases}
\end{align*}
where $g_{r_m}(p)$ is uniformly bounded for $p$ in a compact subset of $\Omega$ -- see \S~\ref{ss.ball}. Therefore, for the isoperimetric term to be stronger than the nonlocal one, the regimes to be considered are \begin{align*}
\gamma\,r_m^3\,|\log r_m|<<1\quad\text{for}\quad n=2,\\
\gamma\,r_m^3<<1\quad\text{for}\quad n\geq3.
\end{align*}
Note that if $\gamma\to0$ the conditions above are trivially satisfied.
On the other hand, in the most interesting case of $\gamma\geq C>0$ one is forced to consider the small volume-fraction regime $r_m<<1$, which we will always assume.
Under these scalings we provide a detailed analysis of the minimizers of $F_{\gamma, m}$, showing that a single droplet is a minimizer for $F_{\gamma, m}$.
In particular, we prove:
\begin{itemize}
\item[(a)] the asymptotic convergence of the minimizers to round spheres in strong norms, providing the rate of convergence;
\item[(b)] the asymptotic optimal centering of the droplet in the domain;
\item[(c)] the expansion of the energy in terms of the radius $r_m$;
\item[(d)] the nonexistence of exact spherical droplets in domains $\Omega$ different from a ball; and, on the other hand, the uniqueness of the minimizer when $\Omega$ is a ball (in this case the minimizers is itself a ball centered at the center of $\Omega$).
\end{itemize}

These results are summarized in the following theorem (see next sections for more details on the notation).

\begin{theorem}\label{t.main-0}
Let $\Omega\subset\R^n$ be a bounded open set with $C^2$ boundary.
There exist $\delta_0, r_0>0$ (depending on $\Omega$) such that the following holds.
Assume $r_m\leq r_0$ and
\[
\gamma\,r_m^3\,|\log r_m|<\delta_0 \quad\text{if }\; n=2
\quad\text{and}\quad
\gamma\,r_m^{3}<\delta_0 \quad \text{if }\; n\geq3.
\]
Then, every minimizer $u_m=\chi_{E_m}\in \cC_m$ of $F_{\gamma,m}$ satisfies the following properties:
\begin{itemize}
\item[(i)] $E_m$ is a convex set and there exists $p_m\in\Omega$ and $\ph_m:\s^{n-1}\to \R$ such that $\de E_m=\{p_m+(r_m+\ph_m(x))\,x:x\in\s^{n-1}\}$, for some $p_m\in\Omega$ and $\ph_m:\s^{n-1}\to \R$ and
\begin{equation}\label{e.rate-0}
\|\ph_m\|_{C^{1}}\lesssim\gamma\,r_m^{n+3};
\end{equation}
%with $C=C(n,\alpha)>0$;
\item[(ii)] $p_m$ is close to the set of harmonic centers $\cH$ of $\Omega$, i.e.~$\lim_{r_m\to0}\dist(p_m,\cH)=0$;
\item[(iii)] the energy of $u_m$ has the following asymptotic expansion:
\begin{equation}\label{e.energy-0}
F_{\gamma,m}(u_m)=
\begin{cases}
2\,\pi\,r_m+\frac{\pi\,\gamma}{2}\,r_m^4\,\log r_m+\gamma\left(-\frac{1}{8}+\pi^2\,\min_\Omega h\right)\,r_m^4+O(r_m^6),& n=2,\\
n\,\omega_n\,r_m^{n-1}+\frac{2\,\gamma\,\omega_n}{4-n^2}\,r_m^{n+2}+\gamma\;\omega_n^2\,r_m^{2n}\,\min_\Omega h+O(r_m^{2n+2}),& n\geq3,
\end{cases}
\end{equation}
where $h$ is the Robin function;
\item[(iv)] $E_m$ is an exact ball if and only if the domain $\Omega$ is itself a ball, i.e.~up to translations $\Omega=B_R$ for some $R>0$, in which case $E_m=B_{r_m}$ is the unique minimizer.
\end{itemize}
\end{theorem}

Many of the mathematical difficulties in Theorem~\ref{t.main-0} are due to our choice to work in any dimension $n$ (previous results are mostly in dimensions $n=2,3$), and with the  standard Coulombian kernel and the natural Neumann boundary condition.
Results analogous to ours have been obtained under different simplified assumptions in \cite{ACO, ChPeI, ChPeII, FiMa, KnMu, Mu, Os, ReWe07, RenWei08}. 
More in detail, in the remarkable paper by Alberti, Choksi and Otto \cite{ACO} the authors study the uniform distribution of the energy and of the order parameter of the minimizers of $F_{\gamma,m}$ (see \cite{Sp} for analogous results in the case of the Ohta--Kawasaki functional ${\mathcal E}_{\eps,\sigma}$). In \cite{Mu} Muratov studies the shape of the minimizers and the asymptotic expansion of the energy in the case of multiple droplet patterns in two dimensions for a slightly different nonlocal interaction and periodic boundary conditions. The author consider a regime where the minimizers of $F_{\gamma,m}$ are given by the union of nearly spherical droplets whose centers minimize a pairwise interaction energy (analogous results are also proved for the functional ${\mathcal E}_{\eps,\sigma}$). In a recent preprint by Kn\"upfer and Muratov \cite{KnMu}, the authors study exact spherical solutions to a $2$-dimensional nonlocal isoperimetric problem in the whole space, where the nonlocal term is a Coulombian interaction. In \cite{ChPeI, ChPeII} Choksi and Peletier study the regime of finitely many (independent of the parameters) droplets proving a $\Gamma$-convergence expansion of $F_{\gamma,m}$ and ${\mathcal E}_{\eps,\sigma}$. In \cite{ReWe07, RenWei08} Ren and Wei construct explicit equilibria solutions, see \S~\ref{ss.elliptic} for more comments. Finally, for a different problem involving a local perturbation of the isoperimetric term, we mention the result by Figalli and Maggi in \cite{FiMa} where an analogous convergence result to single droplet minimizers is proved. 

\subsection{The role of regularity theory}\label{ss.intro-reg}
Many of the existing results regarding the shape of droplets are restricted to the two-dimensional case. This is partially due to the fact that in two dimensions the isoperimetric confinement is strong enough to allow the use of non-parametric techniques, while in higher dimensions having small perimeter does not even imply boundedness (e.g. consider a very thin tube).

One of the main contributions of this paper is to provide robust arguments to overcome this difficulty. In particular, we exploit a  combination of two facts in the regularity theory of minimal surfaces: the uniform regularity properties of minimizers and the use of the optimal quantitative isoperimetric inequality. More precisely, we show that the minimizers of $F_{\gamma,m}$ are uniform $\Lambda$-minimizers of the perimeter (see Definition \ref{d.lambda min}), thus allowing the uses of non-parametric techniques. On the other hand, once in a  uniform neighborhood of the ball, the optimal quantitative isoperimetric inequality is the natural estimate to get the asymptotic convergence for the single droplet. We remark that, to our knowledge, this is the first time when the sharp exponent $2$ in the quantitative isoperimetric inequality has an essential role, while the only case where the uniform regularity property plays a role in the same spirit is a recent paper by Acerbi, Fusco and Morini \cite{AcFuMo} where the authors study local minimizers for the functional $F_{\gamma,m}$ via second variations.

\subsection{Confined solutions to an elliptic system}\label{ss.elliptic}
Problems similar to the ones we have addressed here have been considered by Oshita \cite{Os} and by Ren and Wei \cite{ReWe07, RenWei08} in dimension $n=2,3$.
Their starting point is somehow different from ours: specifically, they are aimed to the construction of special confined solutions to elliptic systems of the form
\begin{equation}\label{e.Oshita}
\begin{cases}
-\Delta v=\chi_{E}-m & \text{in }\Omega,\\
\nabla v\cdot \nu =0 & \text{on }\de\Omega,\\
\gamma\, v+H_{\de E}=0& \text{on }\de E.
\end{cases}
\end{equation}
Here $\Omega$ is a bounded smooth domain with boundary $\partial \Omega$ having normal $\nu$, $E\subset\Omega$ is an open set whose boundary $\partial E$ is a $C^2$-manifold embedded in $\Omega$, $H_{\de E}$ is its mean curvature and $\gamma>0$ and $m\in(0,1)$ are two parameters. Note that such a system is the Euler--Lagrange equation of the functional $F_{\gamma,m}$ whenever $\chi_E$ is a smooth critical point. By exploiting a Lyapunov-Schmidt reduction procedure, in these papers the authors prove that, for a special range of the parameters $m$ and $\gamma$, there exists a stationary point of $F_{\gamma,m}$ having the shape of a single droplet. More precisely, they start from a harmonic center $p$ of $\Omega$, and by a perturbative argument they are able to find solutions to the equilibrium close to the ball centered at $p$. As a byproduct of Theorem~\ref{t.main-0} we obtaine the existence of such solutions in any space dimension as the (local) minimizers of the associated functional $F_{\gamma,m}$.

A couple of remarks on this regard are in order.
While the analysis made by Oshita is done in the somehow simpler regime $\gamma\to0$, Ren and Wei consider a different range of parameters with respect to the one considered here (in particular, they have a restrictive gap condition). Surprisingly enough, thought the techniques are different, we obtain the same rate of convergence of a single droplet minimizer to the ball in dimension $n=2$ as in \cite{ReWe07} (while we improve the rate in the case of higher dimensions).

\subsection{Variants and possible generalizations}
We point out that the techniques developed in this paper may also be applied to several other related models, which have been considered in the last years. This is because the arguments described in \ref{ss.intro-reg} do not rely strongly on the form of the isoperimetric term neither on that of the nonlocal one, but rather on energy scaling and regularity properties of minimizers.
Among the possible generalizations, we mention the extension of the results of the paper to:
\begin{itemize}
\item[(1)] multiple droplets patterns (work in progress);
\item[(2)] models presenting different Coulombian-type kernels $G$ (such as the screened kernels considered by Muratov in \cite{Mu});
\item[(3)] droplets minimizers for the Ohta--Kawasaki functional \eqref{e.OK};
\item[(4)] anisotropic perimeter functionals (for which the asymptotic shape will be the Wulff shape for the chosen anisotropy -- cp.~with \cite{FiMa});
\item[(5)] \textit{nonlocal perimeters}, as those considered by Carlen et al.~in \cite{Ca}. 
\end{itemize}

\bigskip

The paper is organized as follows.
In \S~\ref{s.notation} we fix the notation used throughout the paper and recall some known preliminary results which will be used in the proof of Theorem~\ref{t.main-0}. In \S~\ref{s.reg} we prove the Lipschitz continuity of the nonlocal part of the energy, deriving the first regularity conclusions such as the almost minimality of the minimizers. Then in the short section \S~\ref{s.torus} we show how this observation leads to the main result of this paper in the simpler case when the natural Neumann boundary condition are replaced by periodic boundary conditions. We postpone the proofs of the general case to \S~\ref{s.round}. In \S~\ref{s.stable} we discuss the existence of perfectly spherical solutions, showing how the regularity plays a role also in the study of the stability. The final Appendix is devoted to the proofs of some estimates on the Green function in \eqref{e.introG} we use through the paper.

%%%%%%%%%%%%%%% Notation and results %%%%%%%%%%%%%%%%%%%%%%%%%%%%%%%
%%%%%%%%%%%%%%%%%%%%%%%%%%%%%%%%%%%%%%%%%%%%%%%%%%%%%%%%%%%%%%%%%%%%
%%%%%%%%%%%%%%%%%%%%%%%%%%%%%%%%%%%%%%%%%%%%%%%%%%%%%%%%%%%%%%%%%%%%
%%%%%%%%%%%%%%%%%%%%%%%%%%%%%%%%%%%%%%%%%%%%%%%%%%%%%%%%%%%%%%%%%%%%
%%%%%%%%%%%%%%%%%%%%%%%%%%%%%%%%%%%%%%%%%%%%%%%%%%%%%%%%%%%%%%%%%%%%
\section{Notation and preliminaries}\label{s.notation}
In what follows $\Omega\subset\R^n$ is a bounded open set with $C^2$ boundary $\de \Omega$. 
For given constants $m\in(0,1)$ and $\gamma>0$, we consider the sharp interface limit of the Ohta--Kawasaki functional $F_{\gamma, m}$ which can be written in the following way:
\begin{equation}\label{e.OK-sharp}
F_{\gamma,m}(u):=\int_{\R^n} |D u|+\gamma\int_\Omega\int_\Omega G(x,y)\,\big(u(x)-m\big)\big(u(y)-m\big)\,dx\,dy.
\end{equation}

In the above expression, the order parameter $u$ belongs to the class $\cC_m(\Omega)$ (often we will simply write $\cC_m$) of functions with bounded variation taking values in $\{0,1\}$, whose average in $\Omega$ is $m$ and which are constantly $0$ outside $\Omega$:
\begin{equation}\label{e.Cm}
\cC_m:=\left\{u\in BV(\R^n,\{0,1\})\,:\,\fint_\Omega u=m,\quad u\vert_{\R^n\setminus\Omega}=0\right\}.
\end{equation}

As already noticed in the introduction, we stress once again that in the functional $F_{\gamma, m}$ the total variation of $Du$ is computed in the whole $\R^n$, thus accounting also for possible concentration of this measure (interfaces of the physical system) on the boundary of $\Omega$. In the second term in \eqref{e.OK-sharp}, $G$ denotes the Green function of the Laplacian with Neumann boundary conditions on $\de\Omega$. Denoting by $\nu$ the exterior normal to $\de \Omega$ and by $|A|$ the $n$-dimensional Lebesgue measure of a set $A$, $G$ is defined by the following boundary value problem: for every $x\in\Omega$,
\begin{equation}\label{e.G}
\begin{cases}
-\Delta G(x,\cdot)=\delta_x -\frac{1}{|\Omega|} & \text{in }\Omega,\\
\nabla G(x,\cdot)\cdot \nu =0 & \text{on }\de\Omega,\\
\int_{\Omega}G(x,y)\,dy=0.
\end{cases}
\end{equation}

In place of the average $m$, we will often make use of the parameter $r_m$ corresponding to the radius of a ball whose volume fraction in $\Omega$ is $m$, i.e.
\[
\omega_n\,r_m^n:=m\,|\Omega|.
\]
Moreover, we will often identify $u\in\cC_m$ with the set of finite perimeter $E$ such that $u=\chi_E$ (see \cite{AFP, Gi}). We will write the energy $F_{\gamma, m}$ depending on $E$ in the following way:
\[
F_{\gamma,m}(E):=\Per(E)+\gamma\,\NL(E),
\]
where $\Per(E)=\int_{\R^n}|D\chi_E|$ is the perimeter of $E$ in $\R^n$ and NL is the nonlocal part of the energy. Note that, thanks to $\int_\Omega G(x, y)\,dy=0$ for every $x\in \Omega$, we may rewrite the nonlocal term as
\begin{align}\label{e.NL}
\NL(E):
=\int_\Omega\int_\Omega G(x,y)\,\chi_E(x)\,\chi_E(y)\,dx\,dy.
\end{align}

Finally, we fix the following convention regarding the constants we use in the formula.
Every time we will use the letter $C$ for a constant, this is assumed to be positive and depending only on the dimension of the space $n$ and the domain $\Omega$.
When possible, we will use the simbols $a\lesssim b$, $a\gtrsim b$ and $a\simeq b$ for $a\leq C\,b$, $a\geq C\,b$ and $C^{-1}\leq a\leq C\,b$, respectively.
When we do need to keep track of the constants, we will number them accordingly.

\subsection{Robin function and harmonic centers}
Here we recall some basic facts about the Green function $G$.
Let $\Gamma$ be the fundamental solution of the Laplacian, i.e.
\begin{equation}\label{e.Gamma}
\Gamma(t):=
\begin{cases}
\displaystyle\frac{\log t}{2\pi} & \text{if}\;n=2,\\
\displaystyle\frac{t^{2-n}}{n(2-n)\omega_n} & \text{if}\;n\geq3,
\end{cases}
\end{equation}
and define the regular part $R$ of the Green function in \eqref{e.G} as
\[
R(x,y):=G(x,y)+\Gamma(|x-y|).
\]

We recall that, even if in principle $R$ is not defined in $y=x$, nevertheless, for every $x\in\Omega$, $R(x,\cdot)$ solves the following boundary value problem:
\begin{equation}\label{e.robin}
\begin{cases}
\Delta R(x,\cdot)=\frac{1}{|\Omega|} & \text{in }\Omega,\\
\nabla R(x,\cdot)\cdot \nu = \nabla \Gamma(|x-\cdot|)\cdot \nu & \text{on }\de\Omega.
\end{cases}
\end{equation}
This implies that $R(x,\cdot)$ is an analytic function in the whole $\Omega$ and, hence, it makes sense to consider its extension in $y=x$:
\[
h(x):=R(x,x).
\]
The function $h$ is called the Robin function. As it can be easily seen from\eqref{e.robin} $h$ turns out to be analytic in $\Omega$.

Several estimates on the regular part of the Green function and on the Robin function will play an important role in the identification of the concentration points for the minimizers of $F_{\gamma, m}$. The following facts will be used in the proofs: there exists $r_0$ depending only on $\Omega$ such that, for all $r\leq r_0$
\begin{equation}\label{e.Robin}
|R(x,y)|\simeq |\Gamma(r)| \quad \forall\;x,\ y\ :\ \dist(x,\de\Omega)+\dist(y,\de\Omega)\simeq r.
\end{equation}
Moreover, from \eqref{e.Robin} we deduce also:
\begin{gather}
|G(x,y)|\lesssim -\,\Gamma(|x-y|)+1,\quad\forall\;x,y\in\Omega\label{e.robin0}\\
h(x)\simeq \left|\Gamma\big(\dist(x,\de\Omega)\big)\right|,\quad\forall\; x\in\Omega\setminus\Omega_{r_0},\label{e.robin2}
\end{gather}
where, for every $r>0$, we denote by $\Omega_{r}$ the complement in $\Omega$ of the closed $r$-neigborhood of $\de\Omega$, that is 
\begin{equation}\label{e.Omega r}
\Omega_{r}:=\{x\in\Omega\,:\,\dist(x,\de\Omega)>r\}.
\end{equation}
These estimates are well-known (see for example \cite{Fl} in the case of Dirichlet boundary condition). For readers' convenience we prove these estimate in the Appendix \ref{a.Robin}. From the regularity of $h$ and \eqref{e.robin2}, it follows that $h$ is bounded from below.
In particular, since $h$ is analytic and blows up on the boundary of $\Omega$ (hence, in particular it has no constant directions), it follows that the set of minimum points of $h$ is finite: we denote this set by $\cH$ and call them the \textit{harmonic centers} of $\Omega$.
% 
% %%%%%%%%%%%%%%% Preliminaries %%%%%%%%%%%%%%%%%%%%%%%%%%%%%%%%%%%%%%
% %%%%%%%%%%%%%%%%%%%%%%%%%%%%%%%%%%%%%%%%%%%%%%%%%%%%%%%%%%%%%%%%%%%%
% %%%%%%%%%%%%%%%%%%%%%%%%%%%%%%%%%%%%%%%%%%%%%%%%%%%%%%%%%%%%%%%%%%%%
% %%%%%%%%%%%%%%%%%%%%%%%%%%%%%%%%%%%%%%%%%%%%%%%%%%%%%%%%%%%%%%%%%%%%
% %%%%%%%%%%%%%%%%%%%%%%%%%%%%%%%%%%%%%%%%%%%%%%%%%%%%%%%%%%%%%%%%%%%%
% \section{Preliminaries}\label{s.prel}

\subsection{The quantitative isoperimetric inequality}
The classical isoperimetric inequality states that the perimeter of any measurable set $E$ is bigger than the perimeter of a ball $B_E$ having the same volume as $E$:
\begin{equation}\label{e.iso}
\Per(E)- \Per(B_E)\geq 0,
\end{equation}
with equality only in the case $E$ is itself a ball. Quantitative versions of \eqref{e.iso}, also called {\it Bonnesen-type inequalities} \cite{Osserman79}, have been widely studied (see, for instance, \cite{Hall92, HalHayWei91}). The following, called sharp quantitative isoperimetric inequality, has been proved in \cite{CicLeo10, FiMaPr, FuMaPr}.

\begin{proposition}[Sharp quantitative isoperimetric inequality]\label{p.qip}
There exists a dimensional constant $C=C(n)>0$ such that for every set $E\subset\R^n$ of finite measure, it holds
\begin{equation}\label{e.qip}
C\,\asym(E)^2\leq \frac{\Per(E)-\Per(B_E)}{\Per(B_E)},
\end{equation}
where $\asym(E)$ is the {\it Frankel asymmetry} of $E$ (see \cite{HalHayWei91})
\begin{equation*}
  \asym(E) = \inf\left\{ \frac{|E\sdif (x+B_E)|}{|B_E|},\ x\in
    \R^n\right\}.
\end{equation*}
\end{proposition}
\noindent Here, $V\sdif W = (V\setminus W)\cup (W\setminus V)$ is the symmetric difference between $V$ and $W$. For any given $E\subset\R^n$ measurable set of positive and finite measure, we say that $B_E^{opt}$ is an {\it optimal ball} for $E$ if $|B_E^{opt}|=|E|$ and
\begin{equation*}
  \frac{|E\sdif B_E^{opt}|}{|B_E^{opt}|}=\asym(E). 
\end{equation*} 
The center of an optimal ball will also be referred to as an {\it optimal center}. In general the optimal ball may not be unique but, as proven in \cite[Lemma 6.4]{CicLeo11}, by an elementary application of the Brunn-Minkowsky inequality, whenever $E$ is a strictly convex set the optimal ball is actually unique. 
We finally observe that, denoting by $r$ the radius of $B_E$, \eqref{e.qip} scales in $r$ as follows:
\begin{equation}\label{e.iso2}
|E\sdif B_E^{opt}|^2\lesssim\,r^{n+1}\,\big(\Per(E)-\Per(B_E^{opt})\big).
\end{equation}

\subsection{First variations}\label{ss.first variations}
The first variations of $F_{\gamma, m}$ have been computed for regular sets by Muratov \cite{Mu} in dimension $2$ and $3$, and then in any dimension by Choksi and Sternberg \cite{ChSt}. Given a critical point $E$ of $F_{\gamma,m}$ and $x\in\de E$ a regular point of its boundary, the Euler--Lagrange equation of $F_{\gamma, m}$ at $E$ in $B_r(x)$ is given by:
\begin{equation}\label{e.first var}
H_{\de E}+4\,\gamma\,v=c,
\end{equation}
where $H_{\de E}$ denotes the scalar mean curvature of $\de E$ (namely, $H_{\de E}=\textup{div}\, \nu_{E}$, with $\nu_{E}$ the outer normal to $\de E$), $c\in\R$ is a constant coming from a Lagrange multiplier and $v$ is the solution of the following boundary value problem:
\begin{equation}\label{e.potential}
\begin{cases}
-\Delta v=\chi_E-m & \text{in }\Omega,\\
\nabla v\cdot \nu =0 & \text{on }\de\Omega,\\
\int_\Omega v=0.&
\end{cases}
\end{equation}

Since $\|\chi_E-m\|_{L^\infty}\leq 1$, it follows that $v\in C^{1,\alpha}$ for every $\alpha\in(0,1)$. Therefore, from standard elliptic estimates for the quasilinear mean curvature operator (see \cite{GT}), \eqref{e.first var} implies that, for every critical point $E$, $\de E$ is $C^{3,\alpha}$ for every $\alpha\in(0,1)$ in a neighborhood of a regular point. As shown in the next section, every minimizer of $F_{\gamma, m}$ is regular except a singular set of Hausdorff dimension at most $n-8$ (in particular, it is empty in the physical dimensions $n=2,3$).

\subsection{Asymptotic energy of balls}\label{ss.ball}
Here we give an asymptotic expansion of the energy of small round balls in $\Omega$. Let $\Omega_{r}$ be defined as in \eqref{e.Omega r}. By the regularity assumption on $\de\Omega$, there exists $r_0>0$ such that, for every $r\leq r_0$ and $p\in \Omega_{r}$, the ball $B_r(p)\in\cC_{\omega_n r^n}$. By a direct computation, 
\begin{align}\label{e.exp ball}
F_{\gamma,\omega_nr^n}(B_r(p))&=\Per(B_r(p))+\gamma\NL(B_r(p))\notag\\
&= n\,\omega_n\,r^{n-1}+\gamma\int_{B_r(p)}\int_{B_r(p)}\Gamma(|x-y|)\,dx\,dy+\gamma\int_{B_r(p)}\int_{B_r(p)}R(x,y)\,dx\,dy\notag\\
&=
\begin{cases}
2\,\pi\,r_m+\gamma\,\left(\frac{\pi}{2}\,r_m^4\,\log r+\left(\pi^2\,g_{r_m}(p)-\frac{3\,\pi}{8}\right)\,r_m^4\right), &\text{if }\;n=2,\\
n\,\omega_n\,r^{n-1}+\gamma\, \frac{2\,\omega_n\,r^{2n}}{4-n^2}+\gamma\, g_r(p)\,(\omega_n\,r^{n})^2, &\text{if }\;n\geq 3, 
\end{cases}
\end{align}
where $g_r:\Omega_r\to\R$ is given by
\begin{equation}\label{e.gr}
g_r(p):=\fint_{B_r(p)}\fint_{B_r(p)}R(x,y)\,dx\,dy.
\end{equation}
In the following lemma we show that $g_r$ converges uniformly to the Robin function $h$ as $r\to0$.

\begin{lemma}\label{l.balls}
Let $\Omega\subset\R^n$ be a bounded open set with $C^2$ boundary.
Then, there exists $r_0>0$ such that, for all $r<r_0$, 
\begin{equation}\label{e.g0}
\|g_r-h\|_{L^\infty(\Omega_{r})}\simeq r^2,
\end{equation}
and, for every $r\leq r_0/2$, 
\begin{equation}\label{e.g}
g_r(p)\simeq |\Gamma(\dist(p,\de\Omega))|\quad\forall\;p\in \Omega_{r}\setminus\Omega_{2r}.
\end{equation}
\end{lemma}

\begin{proof}
To show \eqref{e.g0}, let $r_0$ be as in \eqref{e.Robin} and note that, since $R\vert_{\Omega_{r_0}\times \Omega_{r_0}}$ is analytic, we have
\begin{align}\label{e.R exp}
g_r(p)-h(p)&=\fint_{B_r}\fint_{B_r}\left(R(x+p,y+p)-R(p,p)\right)\,dx\,dy\notag\\
&=\fint_{B_r}\fint_{B_r}\left(DR(p,p)(x,y)+\langle D^2R(p,p)\,(x,y),(x,y)\rangle\right)\,dx\,dy+o(r^2)\notag\\
&=r^2\fint_{B_1}\fint_{B_1}\langle D^2R(p,p)\,(x,y),(x,y)\rangle\,dx\,dy+o(r^2),
\end{align}
where in the last equality we used that the linear term integrates to $0$.
By the linearity of the integral and of the scalar product, it follows that 

\begin{align}\label{e.R exp-1}
g_r(p)-h(p)&=\sum_{i,j}\Big(\de_{x_i}\de_{x_j}R(p,p)\,A_{x_i x_j}+2\,\de_{x_i}\de_{y_j}R(p,p)\,A_{x_i y_j}+\de_{y_i}\de_{y_j}R(p,p)\,A_{y_i y_j}\Big),
\end{align}
where 
\begin{align*}
A_{x_i x_i}=A_{y_i y_i}=\mu:=\fint_{B_1}x_1^2\,dx\quad\text{and}\quad A_{x_i x_j}=A_{x_i y_j}=A_{y_i y_j}=0.
\end{align*}
By the simmetry $R(x,y)=R(y,x)$, we infer from \eqref{e.R exp-1} that
\begin{align*}
g_r(p)-h(p)={\rm Tr}\,\big(D^2R(p,p)\big)\, r^2+o(r^2)=
2\,\mu\,\Delta R(p,p)\,r^2+o(r^2)=\frac{2\,\mu\,r^2}{|\Omega|}+o(r^2),
\end{align*}
thus leading to \eqref{e.g0}. The proof of \eqref{e.g} follows from \eqref{e.g0} and \eqref{e.robin2}.
\end{proof}

%%%%%%%%%%%%%%% Regularity %%%%%%%%%%%%%%%%%%%%%%%%%%%%%%%%%%%%%%%%%
%%%%%%%%%%%%%%%%%%%%%%%%%%%%%%%%%%%%%%%%%%%%%%%%%%%%%%%%%%%%%%%%%%%%
%%%%%%%%%%%%%%%%%%%%%%%%%%%%%%%%%%%%%%%%%%%%%%%%%%%%%%%%%%%%%%%%%%%%
%%%%%%%%%%%%%%%%%%%%%%%%%%%%%%%%%%%%%%%%%%%%%%%%%%%%%%%%%%%%%%%%%%%%
%%%%%%%%%%%%%%%%%%%%%%%%%%%%%%%%%%%%%%%%%%%%%%%%%%%%%%%%%%%%%%%%%%%%
\section{Regularity of minimizers}\label{s.reg}
In this section we prove the Lipschitz continuity of the nonlocal term, from which we derive the uniform regularity properties of the minimizers of $F_{\gamma,m}$ in the relevant regimes for the parameters $\gamma$ and $m$.

\subsection{Lipschitz continuity of NL}
Proofs of the Lipschitz continuity of $\NL$ already appeared in the literature (see, for instance, \cite{AcFuMo,Mu, ST}). For our purposes, a more careful quantitative estimate of the Lipschitz constant is necessary.

\begin{proposition}\label{p.quasi min}
For every $\chi_{E_m},\chi_{{G_m}}\in\cC_m$, setting $w=\Gamma*\chi_{G_m}$, it holds
\begin{equation}\label{e.NL cont}
NL(G_m)-NL(E_m)\lesssim (\|w\|_{L^\infty(E_m\sdif {G_m})}+|G_m|)\,|E_m\sdif {G_m}|.
\end{equation}
\end{proposition}

\begin{proof}
We start from \eqref{e.NL} to get
\begin{align}\label{e.NL diff}
\NL({G_m})-\NL(E_m){}={}&\int_\Om\int_\Om {G}(x,y)\big(\chi_{G_m}(x)\,\chi_{G_m}(y)-\chi_{E_m}(x)\,\chi_{E_m}(y)\big)\ dx\ dy\notag\\
={} &\int_\Om\int_\Om {G}(x,y)\chi_{G_m}(x)\,\big(\chi_{G_m}(y)-\chi_{E_m}(y)\big)\ dx\ dy+\notag\\
&+\int_\Om\int_\Om {G}(x,y)\,\chi_{E_m}(y)\big(\chi_{G_m}(x)-\chi_{E_m}(x)\big)\ dx\ dy\notag\\
={} & 2\int_\Om\int_\Om {G}(x,y)\,\chi_{G_m}(x)\,\big(\chi_{G_m}(y)-\chi_{E_m}(y)\big)\ dx\ dy-&\notag\\
&-\int_\Om\int_\Om {G}(x,y)\,\big(\chi_{G_m}(y)-\chi_{E_m}(y)\big)\cdot\notag\\
&\hspace{3.5cm}\cdot\big(\chi_{G_m}(x)-\chi_{E_m}(x)\big)\ dx\ dy,
\end{align}
where in the last equality we used the symmetry $G(x,y)=G(y,x)$.
Since
\begin{equation*}
\int_\Om\int_\Om {G}(x,y)\,\big(\chi_{G_m}(y)-\chi_{E_m}(y)\big)
\cdot\big(\chi_{G_m}(x)-\chi_{E_m}(x)\big)\ dx\ dy=\int_\Omega|\nabla z(x)|^2\,dx\geq 0,
\end{equation*}
where $z$ solves
\[
\begin{cases}
-\Delta z=\chi_{G_m}-\chi_{E_m} & \text{in }\;\Omega,\\
\nabla z\cdot\nu=0 & \text{on }\;\de \Omega,
\end{cases}
\]
we deduce:
% $|G(x,y)|\leq C\Gamma(|x-y|)$ (see Lemma~\ref{l.green} in the Appendix), 
\begin{align}\label{e.NL diff estimate}
\NL({G_m})-\NL(E_m)&\leq 2\,\int_\Om\int_\Om |G(x,y)|\,\chi_{G_m}(x)\,\big|\chi_{G_m}(y)-\chi_{E_m}(y)\big|\ dx\ dy\\
&\stackrel{\mathclap{\eqref{e.robin0}}}{\lesssim}\int_\Om\int_\Om \left( -\Gamma(|x-y|)+1\right)\,\chi_{G_m}(x)\,\big|\chi_{G_m}(y)-\chi_{E_m}(y)\big|dx\, dy\notag\\
&\lesssim\int_\Om (|G_m|-w(y))\big(\chi_{G_m}(y)-\chi_{E_m}(y)\big) \,dy\notag
\\
&\lesssim (\|w\|_{L^\infty({E_m}\sdif {G_m})}+|G_m|) \,|{E_m}\sdif {G_m}|,\notag
\end{align}
that is \eqref{e.NL cont}.
\end{proof}

A straightforward consequence of Proposition~\ref{p.quasi min} is that, if $E_m$ is a minimizer of $F_{\gamma, m}$, then                                                                                       
\begin{align}\label{e.deficit_1}
\Per(E_m)-\Per({G_m})&\leq\gamma\,\big(\NL({G_m})-\NL(E_m)\big)\notag\\
&\lesssim \gamma\,(\|w\|_{L^\infty(E_m\sdif {G_m})}+|G_m|)\,|E_m\sdif {G_m}|.
\end{align}
By the radial monotonicity of $\Gamma$, it holds
\begin{equation}\label{e.monotone}
\|\Gamma * \chi_{G_m}\|_{L^\infty}\leq\|\Gamma * \chi_{B_m}\|_{L^\infty}=\
\begin{cases}
\frac{r_m^2}{2}\,\left(\frac{1}{2}-\log r_m\right) & \text{if } \;n=2,\\
\frac{r_m^{2}}{2\,(n-2)}& \text{if } \;n\geq3,
\end{cases}
\end{equation}
for every $G_m$ with $|G_m|=|B_{r_m}|$. As a result, for $r_m$ sufficiently small, we have 
\begin{equation}\label{e.monotone-stima}
\|w\|_{L^{\infty}}+|G_m|\lesssim\|\Gamma * \chi_{G_m}\|_{L^\infty}+r_m^n\lesssim
\begin{cases}
\frac{r_m^2}{2}\,\left(\frac{1}{2}-\log r_m\right) & \text{if } \;n=2,\\
\frac{r_m^{2}}{2\,(n-2)}& \text{if } \;n\geq3,
\end{cases}
\end{equation}
Here we have used the direct computations:
\begin{equation}\label{e.radial 2}
\Gamma*\chi_{B_r}(x)=
\begin{cases}
\frac{|x|^2}{4}+\frac{r^2}{2}\,\left(\log r -1\right) & \text{if }\;|x|\leq r,\\
\frac{r^2}{2}\,\left(\log |x|-\frac{1}{2}\right) & \text{if }\;|x|> r,
\end{cases}
\qquad \text{if }\;n=2,
\end{equation}
\begin{equation}\label{e.radial n}
\Gamma*\chi_{B_r}(x)=
\begin{cases}
\frac{|x|^2}{2\,n}+\frac{r^2}{2\,(2-n)} & \text{if }\;|x|\leq r,\\
\frac{r^n}{n\,(2-n)\,|x|^{n-2}} & \text{if }\;|x|> r,
\end{cases}
\qquad \text{if }\;n\geq 3.
\end{equation}
Therefore, as soon as $r_m$ is small enough that $\chi_{B_{r_m}(p)}\in\cC_m$ for some $p\in\Omega$, it follows by the previous two estimates, with $G_m=B_{r_m}(p)$, that
\begin{equation}\label{e.deficit_2}
\Per(E_m)-\Per(B_{r_m}(p)) \lesssim 
\begin{cases}
\gamma\,r_m^2\,|\log r_m|\, |E_m\sdif B_{r_m}(p)| &\text{if }\;n=2,\\
\gamma\,r_m^2\, |E_m\sdif B_{r_m}(p)| &\text{if }\;n\geq3.
\end{cases}
\end{equation}
In particular, by the quantitative isoperimetric inequality \eqref{e.iso2}, there exists an optimal isoperimetric ball $B_{E_m}^{opt}$ for $E_m$ such that
\begin{equation}\label{e.iso2bis}
|E_m\sdif B^{opt}_{E_m}|^2\lesssim r_m^{n+1}\,\big(\Per(E_m)-\Per(B^{opt}_{E_m})\big).
\end{equation}
In the case $\chi_{B^{opt}_{E_m}}\in\cC_m$, gathering together \eqref{e.iso2bis} and \eqref{e.deficit_2}, we have that
\begin{equation}\label{e.1st iso}
|E_m\sdif B^{opt}_{E_m}|\lesssim
\begin{cases}
\gamma\,r_m^{n+3}\,|\log r_m| &\text{if }\;n=2,\\
\gamma\,r_m^{n+3} &\text{if }\;n\geq3.
\end{cases}
\end{equation}

\subsection{Volume constraint}\label{ss.volume}
In order to deduce uniform regularity properties for minimizers, it is convenient to get rid of the volume constraint. This has been done in several contexts by the means of a penalization argument.
To this purpose, let us rescale our sets: set $p_m$ for the barycenter of $E_m$ and
\[
H_m:=(E_m-p_m)/r_m\subset \Omega_m:=(\Omega-p_m)/r_m.
\]
Note that $H_m$ is a minimizer of $F_{\gamma r_m^3,m}$ in $\cC_m(\Omega_m)$.  The following lemma shows that, if $H_m$ is sufficiently close to a given $H\subset\R^n$ well-contained in $\Omega_m$, the volume constraint can be dropped.

\begin{lemma}\label{l.vol constr}
Let $H_m\subset\Omega_m$ be as above for $0<m\leq m_0$ with $m_0$ a given constant. Let $H\subset\Omega_m\subset\R^n$ be a set of finite perimeter such that $\dist(H,\de\Omega_m)\geq 1$ for every $m\in(0,m_0)$.
Then, there exists $\Lambda>0$ with this property: for every $m\in(0,m_0)$, if $|H\sdif H_m|\leq \Lambda^{-1}$, then $H_m$ is a minimizer of $G_{\Lambda,m}$,
\[
G_{\Lambda,m}(E):= F_{\gamma r_m^3,m}(E)+\Lambda\,||E|-\omega_n|,
\]
in the class of all sets $E$ with $|E\sdif H|\leq2\,\Lambda^{-1}$.
\end{lemma}

The proof of the lemma follows from a simple adaptation of the computations in \cite[Section~2]{EsFu} (see also \cite[Proposition~2.7]{AcFuMo}).
We give here only the necessary modifications.

\begin{proof}
The proof is by contradiction. Assume that there exist $\Lambda_h\to+\infty$ with this property: there exist $m_h\in(0,m_0)$, $H_h$ minimizers of $F_h:=F_{\gamma r_{m_h},m_h}$ and $E_h$ minimizers of $G_h:=G_{\Lambda_h,m_h}$ such that:
\begin{itemize}
\item[(a)] $|H_h\sdif H|\leq \Lambda_h^{-1}$;
\item[(b)] $|E_h\sdif H|\leq 2\,\Lambda_h^{-1}$;
\item[(c)] $|E_h|<|H_h|=\omega_n$ (the case $|E_h|>|H_h|$ is analogous);
\item[(d)] $G_{h}(E_h)<F_{h}(H_h)$.
% \begin{equation}\label{e.absurd}
% G_{\Lambda_h}(E_h)<
% % \Per(E_h)+\Lambda_h\,||E_h|-1|<
% F_{h}(H_h).
% \end{equation}
\end{itemize}
Since $E_h\to H$ in $L^{1}(\R^n)$, as in \cite[Proposition~2.7]{AcFuMo}, one can show the existence of suitable deformations $\tilde E_h$ satisfying the following:
\begin{itemize}
\item[(e)] $|\tilde E_h|=|H_h|=\omega_n$;
\item[(f)] $\dist(\tilde E_h,H)<1$ (in particular, $\tilde E_h\subset\Omega_m$);
\item[(g)] there exist $\sigma_h>0$ with $|H_h|-|E_h|\geq c_1(n)\,\sigma_h$ such that
\begin{gather*}
\Per(\tilde E_h)\leq \Per(E_h)\,(1+c_2(n)\,\sigma_h)
\quad\text{and}\quad
|\tilde E_h\sdif E_h|\leq c_3(n)\,\sigma_h\,\Per(E_h),
\end{gather*}
where $c_1, c_2, c_3>0$ are dimensional constants.
\end{itemize}
Hence, we infer that, for $h$ sufficiently large,
\begin{align*}
F_{h}(\tilde E_h)&=G_{h}(\tilde E_h)\\
&\leq G_{h}(E_h)+\big[c_2(n)\,\sigma_h\,\Per(E_h)+C\,|E_h\sdif \tilde E_h|-\Lambda_h\,||E_h|-\omega_n|\big]\\
&\stackrel{\mathclap{(d),(g)}}{<}\;
F_{h}(H_h)+\sigma_h\,\big[c_2\,\Per(E_h)+C\,c_3\,\Per(E_h)-c_1\,\Lambda_h\big]\\
&<F_{h}(H_h),
\end{align*}
where we used the Lipschitz continuity of the nonlocal term, $\Lambda_h\to+\infty$ and the uniform bound on $\Per(E_h)$ implied by:
\[
\Per(E_h)\leq F_{h}(H_h)\leq F_{\gamma r_{m_h}^3,m_h}(B_1)
%\Per(B_1)+\gamma\, r_{m_0}^3\,\NL(B_1)
<+\infty \quad\forall\;h\in\N.
\]
This contradicts the minimizing property of $H_h$ and, hence, proves the lemma.
\end{proof}

\subsection{$\Lambda$-minimizers}
It follows from Lemma~\ref{l.vol constr} that the sets $H_m$ are uniform strong $\Lambda$-minimizer of the perimeter according to the following definition.

\begin{definition}\label{d.lambda min}
Let $\Omega\subset\R^n$ be open. A set of finite perimeter $E\subset\Omega$ is a \textit{strong $\Lambda$-minimizer} of the perimeter in $\Omega$ at scale $\eta>0$ if
\begin{equation}\label{e.L min}
P(E) \leq P(F) + \Lambda |E\sdif F|,\quad\forall\; E\sdif F\subset\subset \Omega,\;|E\sdif F|\leq \eta.
\end{equation}
\end{definition}

This notion of almost minimality has been widely studied in the regularity theory for minimal surfaces.
% We say that the sets $E_k\subset\Omega_k$ are \textit{uniformly} strong $\Lambda$-minimizers if \eqref{e.L min} holds for every $k$ with the same $\Lambda,\eta$.
By the theory developed in \cite{Alm, Bom, GMT, Tam} (in particular, see \cite[Theorem~1.9 and Proposition~3.4]{Tam}), strong $\Lambda$-minimizers have regularity estimates which are uniform in the parameters $\Lambda$ and $\eta$. More precisely, for every $\alpha\in(0,1)$, there exists a constant $\eps_0=\eps_0(n,\alpha, \Lambda, \eta)$ such that
\begin{equation*}%\label{e.exc}
\text{Exc}(E, B_{r}(x)) :=  r^{1-n}\left(\Per(E)-|D\chi_E (B_r(x))|\right)\leq \eps_0 \quad \Longrightarrow \quad \de E\cap B_{\frac{r}{2}}(x)\in C^{1,\alpha}.
\end{equation*}
Since the quantity Exc is continuous under $L^1$ convergence of $\Lambda$-minimizers, uniform regularity estimates can be inferred for $\Lambda$-minimizers in a neighborhood of a given smooth set.
% 
% In the following proposition we collect the results we will need in the sequel of the paper (see \cite{Alm, Bom, GMT} and  \cite[Theorem~1.9 and Proposition~3.4]{Tam}) -- according to the discussion in \S~\ref{ss.volume}, we need here to consider the case of varying domains $\Omega_k$, but the proofs remain unchanged.

\begin{proposition}\label{p.regularity}
Let $E\subset\Omega$ be a strong $\Lambda$-minimizers at scale $\eta$ and let $F\subset\R^n$ be a set with smooth boundary and $\dist(F,\de\Omega)\geq 1$. Then, for every $\alpha\in(0,1)$, there exist constants $\eta_0=\eta_0(n,\alpha, \Lambda,\eta) >0$, $R=R(n, \Lambda,\eta)>0$, $c=c(n)>0$  and a modulus of continuity $\omega:\R^+\to\R^+$ with this property:
\begin{itemize}
\item[(i)] if $|E\sdif F|\leq\eta_0$, then $\de E$ can be parametrized on $\de F$ by a function $\ph:\de F\to \R$,
\[
\de E=\big\{x+\ph(x)\,\nu_F(x)\,:\,x\in\de F\big\},
\]
with $\|\ph\|_{C^{1,\alpha}}\leq\omega(|E\sdif F|)$;
\item[(ii)] for all $x\in E$ and $0<r<R$ with $B_r(x)\subset\Omega$, it holds
\begin{equation}\label{e.density}
c(n)\,r^n\leq |E\cap B_r(x)|.
\end{equation}
\end{itemize}
\end{proposition}

Note that, in the sequel, we will apply Proposition~\ref{p.regularity} always in the case $F=B_1$.

\subsection{Higher regularity}\label{ss.improved regularity}
Thanks to Proposition~\ref{p.regularity}, the first variations in \S~\ref{ss.first variations} can be used to improve the regularity of the minimizers of $F_{\gamma,m}$.

\begin{proposition}\label{p.improved}
Let $E_m$ be a minimizer of $F_{\gamma, m}$ and let $H_m$, $p_m$ and $\Omega_m$ be as in \S~\ref{ss.volume}. Then, for every $\alpha \in (0,1)$, there exists $\eta>0$ such that, if $\gamma r_m^3\lesssim 1$, $|H_m\sdif B_1|\leq \eta$ and $\dist(B_1,\de\Omega_m)\geq 1$, then $H_m$ can be parametrized on $\de B_1$,
\[
\de H_m=\big\{(1+\ph_m(x))\,x\,:\,x\in\de B_1\big\}, 
\]
and $\|\ph_m\|_{C^{3,\alpha}}\leq \bar\omega(|H_m\sdif B_1|)$ for a given modulus of continuity $\bar \omega$.
\end{proposition}

\begin{proof}
The existence of a parametrization $\ph_m$ is guaranteed by Proposition~\ref{p.regularity} (i), under the hypothesis that $\eta$ is chosen sufficiently small.
We need only to show that the Euler--Lagrange equation for $F_{\gamma, m}$, namely
\begin{equation}\label{e.first H}
H_{\de H_m}(x+\ph_m(x)\,x)+4\,\gamma\,r_m^3\,w_m(x+\ph_m(x)\,x)=\lambda_m,
\end{equation}
with $\lambda_m\in\R$ a Lagrange multiplier, allows actually to infer the higher regularity claimed in the statement.
Note that it suffices to prove $\sup_{m,\gamma}\|\ph_m\|_{C^{3,\alpha}}\leq C$ (note that the parameter $\gamma$ is hidden in the notation $\ph_m$). Indeed, since $\|\ph_m\|_{C^{1,\alpha}}\leq\omega(|H_m\sdif B_1|)\to 0$ as $\eta\to0$, where $\omega$ is the modulus of continuity in Proposition~\ref{p.regularity}, by compactness in the $C^{3,\alpha}$ norm we would as well deduce that $\|\ph_m\|_{C^{3,\alpha}}\to 0$.

To show this, we consider separately the two different terms in \eqref{e.first H}.
For what concerns $\lambda_m$ we recall that, by Lemma~\ref{l.vol constr} there exists $\Theta>0$ such that $H_m$ minimize $G_{\Theta,m}$ locally in a neighborhood of $B_1$. This allows us to compute the first variations of $G_{\Theta,m}$. Since the penalization term $\Theta\,||E|-\omega_n|$ is not differentiable, we have to distinguish between the variations increasing the volume and those decreasing it. Let $\psi\in C^\infty(\de B_1)$ and $K_\eps$ be the competitor
\[
\de K_\eps:=\big\{x+(\ph_m(x)+\eps\,\psi(x))\,x\,:\,x\in\de B_1\big\}.
\]
The volume of $K_\eps$ is given by
\[
|K_\eps|=\omega_n\int_{\de B_1}(1+\ph_m+\eps\,\psi)^n\,d\cH^{n-1},
\]
% Since from Proposition~\ref{p.regularity} (i) we can assume $\|\ph_m\|_{L^\infty}<1$ and
% \[
% |K_\eps|=\omega_n+\eps\,\omega_n\int_{\de B_1}n\,(1+\ph_m)^{n-1}\,\psi\,d\cH^{n-1}+o(\eps^2),
% \]
hence, it follows that $|K_\eps|> \omega_n$ or $|K_\eps|< \omega_n$ for small $\eps>0$ if $\int\psi>0$ or $\int\psi<0$, respectively.
The minimizing property of $H_m$ implies the following variational inequality to hold true:
\begin{equation*}%\label{e.ineq}
\frac{d G_{\Theta, m}(K_\eps)}{d\eps}\vert_{\eps=0^+}\geq0.
\end{equation*}
In turns this leads to (with analogous computations for the first variations of $F_{\gamma, m}$ as in \cite{ChSt})
\begin{align}
\int_{\de B_1}\Big(H_{\de H_m}(x+\ph_m(x)\,x)+4\,\gamma\,r_m^3\,w_m(x+\ph_m(x)\,x)+\Theta\Big)\,\psi(x) \geq 0\quad \text{if }\;\int_{\de B_1}\psi\,d\cH^{n-1}>0,\label{e.multiplier1}\\
\int_{\de B_1}\Big(H_{\de H_m}(x+\ph_m(x)\,x)+4\,\gamma\,r_m^3\,w_m(x+\ph_m(x)\,x)-\Theta\Big)\,\psi(x) \geq 0\quad \text{if }\;\int_{\de B_1}\psi\,d\cH^{n-1}<0,\label{e.multiplier2}
\end{align}
where $w_m$ solves the boundary value problem:
\begin{equation}\label{e.w_m}
\begin{cases}
-\Delta w_m = \chi_{H_m}-m & \text{in }\;\Omega_m,\\
\nabla w_m\cdot \nu=0 & \text{on }\;\de\Omega_m,\\
\int_{\Omega_m}w_m=0. &
\end{cases}
\end{equation}
Hence, from \eqref{e.multiplier1} and \eqref{e.multiplier2} we deduce a uniform bound on the Lagrange multipliers $\lambda_m$:
\begin{equation}\label{e.multiplier}
|\lambda_m|\leq \Theta\quad\forall m>0.
\end{equation}
For what concerns $w_m$, by an analogous computation as in \eqref{e.monotone} using $|G|\lesssim |\Gamma|+1$ and the radial monotonicity of $\Gamma$, we deduce that $\|w_m\|_{L^\infty}\leq C$. Moreover, since $\|\chi_{H_m}-\chi_{B_1}\|_{L^p}\lesssim\eta$ for every $p>n$, the Sobolev embendings and the Gagliardo--Niremberg interpolation inequality leads to uniform $W^{2,p}$ bounds and, hence, $C^{1,\alpha}$ bounds on $w_m$ for every $\alpha\in(0,1)$ (see \cite[Chapter 9]{Bre}). 

Hence, since $\ph_m$ has also uniform $C^{1,\alpha}$ bounds, the non-parametric theory for the mean curvature-type equation \eqref{e.first H} (see \cite[Chapter 3]{LU} or \cite[Appendix C]{Gi}) finally yelds the desired uniform $C^{3,\alpha}$ estimates for $\ph_m$. 
\end{proof}

%%%%%%%%%%%%%%%%%%%%%%% Torus %%%%%%%%%%%%%%%%%%%%%%%%%%%%%%%%%%%%%%%%
%%%%%%%%%%%%%%%%%%%%%%%%%%%%%%%%%%%%%%%%%%%%%%%%%%%%%%%%%%%%%%%%%%%%%%
%%%%%%%%%%%%%%%%%%%%%%%%%%%%%%%%%%%%%%%%%%%%%%%%%%%%%%%%%%%%%%%%%%%%%%
%%%%%%%%%%%%%%%%%%%%%%%%%%%%%%%%%%%%%%%%%%%%%%%%%%%%%%%%%%%%%%%%%%%%%%
%%%%%%%%%%%%%%%%%%%%%%%%%%%%%%%%%%%%%%%%%%%%%%%%%%%%%%%%%%%%%%%%%%%%%%
%%%%%%%%%%%%%%%%%%%%%%%%%%%%%%%%%%%%%%%%%%%%%%%%%%%%%%%%%%%%%%%%%%%%%%
\section{Periodic boundary conditions: $\Omega=\T^n$}\label{s.torus}

In this section, in order to show the proof of our main result in a technically simpler case, we prove a statement analogous to Theorem~\ref{t.main-0} for periodic boundary conditions. Indeed, since in the present case one discards the interactions with the boundary and the optimal centering of the asymptotic droplet, the proof is a direct consequence of the regularity arguments developed in the previous section.

\subsection{Notation and statement}
Let $\T^n$ be the $n$-dimensional torus obtained as the quotient of $\R^n$ via the $\mathbb{Z}^n$ lattice or, equivalently, $\T^n:=[0,1]^n$ with the identification of opposite faces.
We consider functions
\[
u\in BV(\T^n;\{0,1\})\quad\text{with}\quad \fint_{\T^n}u=m.
\]
As usual such functions $u$ can be identified with measurable sets $E\subseteq\R^n$ invariant under the action of $\mathbb{Z}^n$ and such that $|E\cap[0,1]^n|=m$.
The confining term of the energy is then given by the perimeter of $E$ in the torus:
\[
\Per(E,\T^n):=\int_{[0,1)^n}|D\chi_E|;
\]
and the nonlocal term by:
\[
\NL(E):=\int_{[0,1]^n}\int_{[0,1]^n}G(x,y)\,\chi_E(x)\,\chi_E(y)\,dx\,dy,
\]
where $G$ is the Green function for the Laplacian in $\T^n$, i.e.
\begin{equation}\label{e.G per}
\begin{cases}
-\Delta G(x,\cdot)=\delta_x -1 & \text{in }\T^n,\\
\int_{\Omega}G(x,y)\,dy=0.
\end{cases}
\end{equation}
By the invariance of the torus, we can write with a sligth abuse of notation $G(x,y)=G(|x-y|)$. In the case of periodic boundary conditions, Theorem~\ref{t.main-0} reduces to a statement regarding the shape of the minimizers $E_m$ and the asymptotic behavior of the energy.

\begin{theorem}\label{t.periodic}
There exists $\delta_0>0$ such that the following holds. Assume $r_m<1$ and
\begin{equation*}
\gamma\,r_m^{3}|\log r_m|<\delta_0 \quad\mbox{ if }\quad n=2\quad\mbox{ and }\quad\gamma\,r_m^{3}<\delta_0\quad\mbox{ if }\quad n\geq 3.
\end{equation*}
Then, every $E_m\subset\T^n$ minimizer of $F_{\gamma,m}$ is, up to a translation, a convex set such that
% $\de E_m=\{p_m+(r_m+\ph_m(x))\,x:x\in\s^{n-1}\}$,
\begin{equation*}
\de E_m=\{(1+\psi_m(x))\,r_m\,x:x\in\s^{n-1}\},
\end{equation*}
for some $\psi_m:\s^{n-1}\to \R$ with
\begin{equation}\label{e.rate periodic}
\|\psi_m\|_{C^1}\lesssim\gamma\,r_m^{n+3},
\end{equation}
and its energy has the following asymptotic expansion:
\begin{equation}\label{e.energy periodic}
F_{\gamma,m}(\chi_{E_m})=
\begin{cases}
2\,\pi\,r_m+\frac{\pi\,\gamma}{2}\,r_m^4\,\log r_m+\gamma\left(-\frac{1}{8}+\pi^2\,h(0)\right)\,r_m^4+O(r_m^6),& n=2,\\
n\,\omega_n\,r_m^{n-1}+\frac{2\,\gamma\,\omega_n}{4-n^2}\,r_m^{n+2}+\gamma\;\omega_n^2\,r_m^{2n}\,h(0)+O(r_m^{2n+2}),& n\geq3.
\end{cases}
\end{equation}
\end{theorem}

\subsection{Improved perimeter estimate}
Due to the translation invariance, for every minimizer $E_m$ we may assume that $B^{opt}_{E_m}=B_{r_m}$ is centered at the origin. Therefore, from \eqref{e.1st iso} we infer for $H_m=E_m/r_m$ that
\[
|H_m\sdif B_1|\lesssim 
\begin{cases}
\gamma\,r_m^3\,|\log r_m|<\delta_0&\text{if }\;n=2,\\
\gamma\,r_m^3<\delta_0&\text{if }\;n\geq3.
\end{cases}
\]
If $\delta_0$ is chosen sufficiently small, by Lemma~\ref{l.vol constr}, the sets $H_m$ are minimizer of some penalized functional $G_{\Lambda,m}$ and, hence, are $\Lambda$-minimizers of the perimeter. By Proposition~\ref{p.regularity}, $H_m$ can be parametrized by the graph of a function $\ph_m$ on $\de B_{1}$ satisfying
\begin{equation*}%\label{e.psi}
\|\ph_m\|_{L^\infty(\de B_{1})}\lesssim |H_m\sdif B_{1}|,
\end{equation*}
thus implying that $E_m$ can be parametrized on $\de B_{r_m}$ by $\psi_m$ with
\begin{equation}\label{e.psi}
\|\psi_m\|_{L^\infty(\de B_{r_m})}\lesssim \frac{|E_m\sdif B_{r_m}|}{r_m^{n-1}}.
\end{equation}
These observations lead to the following proposition which is a consequence of an improved estimate for the Lipschitz constant of the nonlocal part of the energy.

\begin{proposition}\label{p.deficit}
There exists $\delta_0>0$ such that, if $\gamma\,r_m^3\,|\log r_m|\leq\delta_0$ in the case $n=2$ and if $\gamma\,r_m^3\leq\delta_0$ in the case $n\geq3$, and $E_m$ is a minimizer of $F_{\gamma,m}$, then
\begin{equation}\label{e.deficit periodic}
\Per({E_m})-\Per(B^{opt}_{E_m})\lesssim \gamma\,\frac{|{E_m}\sdif B^{opt}_{E_m}|^2}{ r_m^{n-2}\,}+\gamma\,r_m^{n+1}\,|{E_m}\sdif B^{opt}_{E_m}|.
\end{equation}
\end{proposition}

\begin{proof}
Recalling \eqref{e.NL diff estimate} in Proposition~\ref{p.quasi min}, and assuming as above $B_{r_m}=B^{opt}_{E_m}$, we have that
\begin{align*}%\label{e.NL diff estimate}
\NL(B_{r_m})-\NL(E_m)\lesssim
%\stackrel{\eqref{e.NL diff estimate}}{\lesssim}
&{} \int_\Om\int_\Om {G}(x,y)\,\chi_{B_{r_m}}(x)\,\big(\chi_{B_{r_m}}(y)-\chi_{E_m}(y)\big)\ dx\ dy\notag\\
=&{}\int_\Om\int_\Om \Gamma(|x-y|)\,\chi_{B_{r_m}}(x)\,\big(\chi_{B_{r_m}}(y)-\chi_{E_m}(y)\big)dx\, dy+\notag\\
&+\int_\Om\int_\Om R(x,y)\,\chi_{B_{r_m}}(x)\,\big(\chi_{B_{r_m}}(y)-\chi_{E_m}(y)\big)dx\, dy.
\end{align*}
By the direct computation of $w=\Gamma*\chi_{B_{r_m}}$ in \eqref{e.radial 2} and \eqref{e.radial n} (in particular, $|\nabla w|\lesssim r_m$ in a neighborhood of $\de B_{r_m}$), we get,
\begin{equation*}%\label{e.wopt}
\|w\|_{L^{\infty}(E_m\sdif {B_{r_m}})}\lesssim r_m\, \|\psi_m\|_{L^\infty(\de {B_{r_m}})}
%\dist(\partial E_m,\partial B_m) 
\stackrel{\eqref{e.psi}}{\lesssim} \frac{|E_m\sdif {B_{r_m}}|}{r_m^{n-2}}.
\end{equation*}
Moreover,
\begin{align*}
&\int_\Om\int_\Om R(x,y)\,\chi_{B_{r_m}}(x)\,\big(\chi_{B_{r_m}}(y)-\chi_{E_m}(y)\big)dx\, dy\notag\\
&=R(0,0)\int_\Om\int_\Om \chi_{B_{r_m}}(x)\,\big(\chi_{B_{r_m}}(y)-\chi_{E_m}(y)\big)dx\, dy+O(r_m^{n+1})|E_m\sdif B_{r_m}|\notag\\
&\simeq r_m^{n+1}\;|E_m\sdif B_m|,
\end{align*}
where we used that $\int\chi_{B_{r_m}}-\int \chi_{E_m}=0$. Hence, gathering together the previous estimates, by the minimality of $E_m$, it follows:
\begin{align*}
\Per(E_m)-\Per({B_{r_m}})&\leq \gamma\,\NL({B_{r_m}})-\gamma\,\NL(E_m)\\
&\simeq \gamma\,\|w\|_{L^\infty({E_m}\sdif {B_{r_m}})} \,|{E_m}\sdif {B_{r_m}}|+\gamma\,r_m^{n+1}|E_m\sdif {B_{r_m}}|\notag\\
&\simeq\gamma\,\frac{|{E_m}\sdif {B_{r_m}}|^2}{ r_m^{n-2}\,}+\gamma\,r_m^{n+1}\,|{E_m}\sdif {B_{r_m}}|.
\end{align*}
\end{proof}

\subsection{Proof of Theorem~\ref{t.periodic}}
In order to prove \eqref{e.rate periodic} we use the quantitative isoperimetric inequality and the improved estimate in Proposition~\ref{p.deficit} to infer that
\begin{align*}
|{E_m}\sdif B_{E_m}^{opt}|^2\;&\stackrel{\mathclap{\eqref{e.iso2}}}{\lesssim}\; r_m^{n+1}\big(\Per({E_m})-\Per(B_{E_m}^{opt})\big)\\
&\stackrel{\mathclap{\eqref{e.deficit periodic}}}{\lesssim} \gamma\,r_m^3\,|{E_m}\sdif B_{E_m}^{opt}|^2+\gamma\,r_m^{2n+2}\,|{E_m}\sdif B_{E_m}^{opt}|.
\end{align*}
This implies $|{E_m}\sdif B_{E_m}^{opt}|\lesssim\gamma\,r_m^{2n+2}$,
% \begin{equation*}
% |{E_m}\sdif B_{E_m}^{opt}|\lesssim\gamma\,r_m^{2n+2},
% \end{equation*}
that is, by \eqref{e.psi},
\begin{equation}\label{e.L1 estimate2}
\|\psi_m\|_{L^1(\de B_1)}\lesssim\gamma\,r_m^{n+3}.
\end{equation}
From the $C^{3,\alpha}$ regularity of $\psi_m$ proved in Proposition~\ref{p.improved}, the convexity of $E_m$ and \eqref{e.rate periodic} follows. Similarly, by comparing the energy of $E_m$ with that of $B_{r_m}$, using Proposition~\ref{p.quasi min}, Proposition~\ref{p.deficit} and Lemma~\ref{l.balls}, \eqref{e.energy periodic} follows:
\begin{align*}
F_{\gamma,m}(E_m)&=F_{\gamma,m}(B_{r_m})+O(\gamma\,r_m^{3n+3})\\
& =
\begin{cases}
2\,\pi\,r_m+\frac{\pi\,\gamma}{2}\,r_m^4\,\log r_m+\gamma\left(-\frac{1}{8}+\pi^2\,h(0)\right)\,r_m^4+O(r_m^6),& \text{if } \ n=2\\
n\,\omega_n\,r_m^{n-1}+\frac{2\,\gamma\,\omega_n}{4-n^2}\,r_m^{n+2}+\gamma\;\omega_n^2\,r_m^{2n}\,h(0)+O(r_m^{2n+2}),& \text{if } \ n\geq3.\hspace{1.8cm}\qed
\end{cases}
\end{align*}
%%%%%%%%%%%%%%%%%%%%%%% Round sphere %%%%%%%%%%%%%%%%%%%%%%%%%%%%%%%%%
%%%%%%%%%%%%%%%%%%%%%%%%%%%%%%%%%%%%%%%%%%%%%%%%%%%%%%%%%%%%%%%%
%%%%%%%%%%%%%%%%%%%%%%%%%%%%%%%%%%%%%%%%%%%%%%%%%%%%%%%%%%%%%%%%
%%%%%%%%%%%%%%%%%%%%%%%%%%%%%%%%%%%%%%%%%%%%%%%%%%%%%%%%%%%%%%%%
%%%%%%%%%%%%%%%%%%%%%%%%%%%%%%%%%%%%%%%%%%%%%%%%%%%%%%%%%%%%%%%%
\section{Strong convergence to round spheres}\label{s.round}
In this section we prove Theorem~\ref{t.main-0} (i), (ii) and (iii). 

We remark that, in this general case, before we may argue as in the proof of Theorem \ref{t.periodic}, we need to show that the minimizers of $F_{\gamma,m}$ are well-contained in $\Omega$. This is crucial in order to apply the regularity results in \S~\ref{s.reg}, which hold under the hypothesis of being at a fixed distance to the boundary. The proof is based on the analysis of the non local energy of a minimizer when it gets close to $\de\Omega$. To this extent a key role will be played by the estimates \eqref{e.Robin}, \eqref{e.robin0} and \eqref{e.robin2}. 

\subsection{Localization of minimizers}
In the next proposition we prove that the minimizers of $F_{\gamma,m}$ are well-contained in $\Omega$. 

\begin{proposition}\label{p.local}
There exist $\delta_0,r_0>0$ such that the following holds. Assume $r_m\leq r_0/3$ and $\gamma\,r_m^3\,|\log r_m|\leq\delta_0$ if $n=2$ and $\gamma \,r_m^3 \leq \delta_0$ if $n\geq3$.
% \begin{equation*}
% \gamma\,r_m^{3}|\log r_m|<\delta_0 \quad\mbox{ if }\quad n=2\quad\mbox{ and }\quad\gamma\,r_m^{3}<\delta_0\quad\mbox{ if }\quad n\geq 3.
% \end{equation*}
Then, every minimizer $E_m$ of $F_{\gamma,m}$ satisfies
\begin{equation}\label{e.local}
E_m\subset B_{3\,r_m}(q)\quad\text{for some }q\in\Omega_{r_0}.
\end{equation}
\end{proposition}

\begin{proof}
We prove the result in the case $n\geq 3$, the case $n=2$ being similar up to minor changes. The proof consists of three steps.

\medskip

\textsc{Step 1.} \textit{If $\delta_0$ and $r_0$ are sufficiently small, then there exists a ball $B_m:=B_{r_m}(p_m)\subset\Omega$ such that
\begin{equation}\label{e.first local}
|B_m\sdif E_m|\lesssim \delta_0^{1/(n+1)}\,r_m^n\quad\text{and}\quad
\dist(p_m,\de\Omega)\gtrsim \delta_0^{-1/((n+1)(n-2))}\,r_m.
\end{equation}
}

For any ball of radius $B_{r_m}(p)\subset\Omega$ (note that such a ball exists if $r_0$ is choosen sufficiently small), by \eqref{e.deficit_2} it holds
\begin{equation}\label{e.deficit_3}
\Per({E_m})-\Per({B_{r_m}(p)})\lesssim \gamma\,r_m^{2}|E_m\sdif B_{r_m}(p)|\lesssim\gamma\,r_m^{n+2}.
\end{equation}
On the other hand, by the quantitative isoperimetric inequality we have
\begin{equation}\label{e.deficit_4}
|E_m\sdif B_{E_m}^{opt}|^2\stackrel{\eqref{e.iso2bis}}{\lesssim} r_m^{n+1}\left(\Per(E_m)-\Per(B_{E_m}^{opt})\right)
\stackrel{\eqref{e.deficit_3}}{\lesssim} \gamma \,r_m^{2n+3}.
\end{equation}
Note that $B_{E_m}^{opt}$ may not be contained in $\Omega$.
Nevertheless, since $B_{E_m}^{opt}\setminus\Omega\subset B_{E_m}^{opt}\sdif E_m$, it follows from \eqref{e.deficit_4} that
\[
|B_{E_m}^{opt}\setminus\Omega|\lesssim (\gamma\,r_m^3)^{1/2}\,r_m^n\lesssim\delta_0^{1/2}\,r_m^n.
\]
We can now use a simple geometric argument proved in Lemma~\ref{l.archimede} below to deduce the existence of a vector $v\in\R^n$ such that
\[
|v|\simeq \delta_0^{1/(n+1)}r_m \quad \text{and} \quad B_m:=B_{E_m}^{opt}+v\subset\Omega.  
\]
Setting $p_m$ as the center of $B_m$, namely $B_m=B_{r_m}(p_m)$, this amounts to say that $p_m\in\Omega_{r_m}$.
We claim that $B_m$ satisfies \eqref{e.first local}.
Note that, since the measure of the symmetric difference between two balls is linear with the distance of the centers, we infer the first conclusion in \eqref{e.first local}, namely
\begin{equation}\label{e.loc 1.1}
|E_m\sdif B_m|\leq |E_m\sdif B_{E_m}^{opt}|+|B_{E_m}^{opt}\sdif B_m|\lesssim \left(\delta_0^{1/2}+\delta_0^{1/(n+1)}\right)\,r_m^n\lesssim \delta_0^{1/(n+1)}\,r_m^n.
\end{equation}
On the other hand, appealing to the minimality of $E_m$ (now $\chi_{B_m}\in\cC_m$) and using \eqref{e.exp ball}, we get:
\begin{align}\label{e.ball vs E}
% F_{\gamma,m}(B_m'')-F_{\gamma,m}(B_{m}^{\min})
\gamma\,(\omega_n\,r_m^{n})^2\,g_{r_m}(p_m)-\gamma\,(\omega_n\,r_m^{n})^2\,\min_{p\in\Omega_{r_m}} g_{r_m}(p)&=F_{\gamma,m}(B_m)-\min_{p\in\Omega_{r_m}} F_{\gamma,m}(B_{r_m}(p))\notag\\
&\leq F_{\gamma,m}(B_m)-F_{\gamma,m}(E_m)
\notag\\
% &= \Per(B_m)-\Per(E_m)+
&\leq\gamma\,NL(B_m)-\gamma\,NL(E_m)
\notag\\
&\lesssim
% \stackrel{\mathclap{\eqref{e.deficit_1},\eqref{e.loc 1.1}}}{\lesssim}
\gamma\,\delta_0^{1/(n+1)}\,r_m^{n+2},
\end{align}
where in the last inequality we have used Proposition \ref{p.quasi min} and \eqref{e.monotone-stima}.
Then, by Lemma \ref{l.balls} and \eqref{e.Gamma} we obtain the second inequality in \eqref{e.first local}:
\begin{equation*}%\label{e.loc 1.2}
\left(\dist(p_m,\de\Omega)+r_m\right)^{2-n}\lesssim \delta_0^{1/(n+1)}\,r_m^{2-n}+\min_{p\in\Omega_{r_m}} g_{r_m}(p)\lesssim \delta_0^{1/(n+1)}\,r_m^{2-n}.
\end{equation*}

\medskip

\textsc{Step 2.} \textit{The whole $E_m$ is well-contained in $\Omega$, i.e.~
\begin{equation}\label{e.localization}
E_m\subset B_{3r_m}(p_m).
%\dist(E_m,\de\Omega)\gtrsim \delta_0^{-1/((n+1)(n-2))}\,r_m.
\end{equation}
}

With the notation as in \S~\ref{ss.volume}, by \eqref{e.1st iso} we have that $|H_m\sdif B_1|\lesssim \gamma r_m^3< \delta_0$. Then, using Lemma~\ref{l.vol constr}, the sequence of sets $H_m$ turns out to be a sequence of uniform $\Lambda$-minimizer of the perimeter in $\Omega_m$. 
%$H_m$ turns out to be a $(\gamma\,r_m^3)$-minimizer of the perimeter at fixed volume in $\Omega$, i.e.
%\begin{equation}\label{e.deficit_5}
%\Per({H_m})-\Per({G})\lesssim \gamma\,r_m^{3}\,|{H_m}\sdif {G}|\quad\forall \; G\subset\Omega_m, \; |G|=|H_m|.
%\end{equation}
Moreover, by \eqref{e.first local}, if $\delta_0$ is small enough, we have that
\begin{equation}\label{e.weak dist}
\dist(0, \de\Omega_m) \gtrsim \delta_0^{-1/(n+1)(n-2))}\geq 4.
\end{equation}
As a consequence, we are in position to use the density estimate \eqref{e.density} in Proposition~\ref{p.regularity} according to which there exists $R>0$ (without loss of generality we assume $R<1$) such that, for every $x\in H_m\cap (B_3\setminus B_2)$,
\[
% c(n)\,R^n\leq c(n)\leq |H_m\cap B_{R}(x)|\leq |H_m\cap B_{1}(x)|.
c(n)\,R^n\leq |H_m\cap B_{R}(x)|\leq |H_m\cap B_{1}(x)|.
\]
Therefore, since for every $x\in H_m\cap (B_3\setminus B_2)$ it holds $B_{1}(x)\cap B_1=\emptyset$, we get:
\[
c(n)\,R^n\leq |H_m\cap B_{1}(x)|\leq |H_m\sdif B_1|\stackrel{\eqref{e.first local}}{\lesssim} \delta_0^{1/(n+1)}.
\]
Clearly, if $\delta_0$ is small enough, this inequality cannot be satisfied, thus implying $H_m\cap (B_3\setminus B_2)=\emptyset$.
% \begin{equation}\label{e.loc 2.1}
% H_m\cap (B_3\setminus B_2)=\emptyset.
% \end{equation}
In order to complete the proof of \eqref{e.localization}, we need to show that $H_m \cap (\Omega_m\setminus B_3)=\emptyset$ as well. To this purpose, we reason by contradiction and show that, in this eventuality, a suitable rescaling of $J_m:=H_m\cap B_2$ would have lower energy than $H_m$.
We fix the notation:
% $K_m:=H_m\setminus B_2=H_m\setminus J_m$, and $L_m:=\rho_m\,J_m$
\[
K_m:=H_m\setminus J_m\quad\text{and}\quad L_m:=\rho_m\,J_m,
\]
with $\rho_m\geq1$ such that $|L_m|=|H_m|$. 
Note first the following two observations:
(a) by a simple computation on the volumes, it follows that
\begin{equation}\label{e.rho_m}
\rho_m-1\lesssim |K_m|;
\end{equation}
(b) consequently, we can estimate $|L_m\sdif J_m|$ in the following way:
\begin{align}\label{e.dilation}
|L_m\sdif J_m|&=\int_{\R^n}|\chi_{J_m}(\rho_m^{-1}x)-\chi_{J_m}(x)|\,dx\notag\\
&\leq \int_{B_3}\int_{0}^{1}|D\chi_{J_m}(sx+(1-s)\,\rho_m^{-1}x)|\,(1-\rho_m^{-1})\,|x|\,ds\,dx\notag\\
&\lesssim(\rho_m-1)\,\Per(J_m)\lesssim |K_m|,
\end{align}
where, in order to rigourously justify the second inequality without referring to fine properties of functions of bounded variations, it is enough to consider an approximation with smooth functions and to pass to the limit.
% \begin{itemize}
% \item[(a)] since $|J_m\sdif B_1|\leq |H_m\sdif B_1|\lesssim \delta_0^{2/(n+1)}$, for $\delta_0$ sufficiently small $J_m$ is a $C^{1,\alpha}$ graph over $B_1$ (otherwise Proposition~\ref{p.regularity} would be contradicted);
% \item[(b)] by a simple computation we deduce that
% \begin{equation}\label{e.rho_m}
% \rho_m-1\lesssim |K_m|.
% \end{equation}
% \end{itemize}
Recalling that $F_{\gamma,m}(E_m)=r_m^{n-1}\,F_{\gamma\,r_m^3,m}(H_m)$, we can compare the energies of $H_m$ and $J_m$ as follows: 
\begin{align}
F_{\gamma\,r_m^3,m}(L_m)=&{}\;\rho_m^{n-1}\,\Per(J_m)+\gamma\,r_m^3\,NL(L_m)\notag\\
\leq&{}\;
\rho_m^{n-1}\,\Per(H_m)-\rho_m^{n-1}\,\Per(K_m)+\gamma \,r_m^3\,C\,|L_m\sdif  H_m|+\gamma\,r_m^3\,NL(H_m)\notag\\
\stackrel{\mathclap{\eqref{e.rho_m}}}{\leq}&{}\;
(1+C\,|K_m|)\,\Per(H_m)-\rho_m^{n-1}\Per(K_m)+\notag\\
&+\gamma \,r_m^3\,C\,\big(|L_m\sdif J_m|+|K_m|\big)+\gamma\,r_m^3\,NL(H_m)\notag\\
\stackrel{\mathclap{\eqref{e.dilation}}}{\leq}&{}\; F_{\gamma\,r_m^3,m}(H_m)+C\,|K_m|+C\,\delta_0\,|K_m|-C\,|K_m|^{(n-1)/n}\notag\\
<&{}\;F_{\gamma\,r_m^3,\omega_n}(H_m),
\end{align}
if $\delta_0$ is sufficiently small since $|K_m|\leq \delta_0^{1/(n+1)}<1$.
Clearly, this is a contradiction with the minimality of $H_m$, thus proving that $H_m\setminus B_3=\emptyset$ or, after scaling by $r_m$, that \eqref{e.localization} holds true.

\medskip

\textsc{Step 3.} \textit{Proof of \eqref{e.local}.}
%Denote by $q_m$ the barycenter of $E_m$ and 
We set $E_m':=E_m-p_m$ and, as a consequence of \eqref{e.localization}, we note that $E_m'\subset B_{3r_m}$. For all $q\in\Omega_{3r_m}$, let us set $E_m(q):=E_m'+q$ (in particular, $E_m(q)\subset\Omega$ and $E_m=E_m(p_m)$). We may write the energy of $E_m(q)$ as
\begin{multline}\label{e.energy100}
F_{\gamma,m}(E(q))=\Per(E_m')+\gamma\,\int\int\Gamma(|x-y|)\,\chi_{E_m'}(x)\,\chi_{E_m'}(y)\,dx\,dy\,+\\
+\gamma\,\int\int R(x+q,y+q)\,\chi_{E_m'}(x)\,\chi_{E_m'}(y)\,dx\,dy.
\end{multline}
Since $E_m$ minimizes $F_{\gamma,m}$, we have that $F_{\gamma,m}(E_m(p_m))\leq F_{\gamma,m}(E_m(q))$ for every $q\in\Omega_{3r_m}$. By \eqref{e.energy100} this implies that
\begin{multline}\label{e.energyNL}
\int\int R(x-p_m,y-p_m)\,\chi_{E_m'}(x)\,\chi_{E_m'}(y)\,dx\,dy\\
\leq
\int\int R(x-q,y-q)\,\chi_{E_m'}(x)\,\chi_{E_m'}(y)\,dx\,dy.
\end{multline}
In view of $E_m'\subset B_{3r_m}$, \eqref{e.Robin} and \eqref{e.Gamma}
%, i.e.
%\[
%R(x,y)\gtrsim (\dist(q,\de\Omega)+r_m)\quad\forall \;x,y\in E_m(q),
%\]
the last inequaltiy \eqref{e.energyNL} implies that $p_m$ is contained in a compact subset of $\Omega$, thus proving \eqref{e.local}.
\end{proof}

\begin{remark}\label{r.optimal ball}
It follows a posteriori that the optimal balls $B_{E_m}^{opt}$ for $E_m$ are, in fact, well-contained in $\Omega$ and \eqref{e.deficit_4} holds, i.e.
\begin{equation}\label{e.deficit100}
\dist(B_{E_m}^{opt},\de\Omega)\gtrsim \delta_0^{-1/((n+1)(n-2))} \quad\text{and}\quad |B_{E_m}^{opt}\sdif E_m|\lesssim \delta_0^{1/2}\,r_m^n.
\end{equation}
\end{remark}

The following is the geometric lemma used in the localization argument.

\begin{lemma}\label{l.archimede}
Let $\Omega\subset\R^n$ be an open set with $C^2$ boundary.
Then, there exist $r_0,h_0>0$ with this property: for $r<r_0$, $h\leq h_0$ and $p\in\Omega$ such that $|B_r(p)\setminus\Omega|\leq h\,r^n$, there exists $v\in\R^n$ with $|v|\lesssim h^{2/(n+1)}r$ such that $B_r(p+v)\subset\Omega$.
\end{lemma}

\begin{proof}
The main argument in the proof is given by an elementary consideration.
Assume first that $\de \Omega \cap B_1(p)$ is flat.
If $|B_1(p)\setminus\Omega|\leq h$ and $h\leq h_0$ is small enough, then $\dist(q,r)\simeq h^{2/(n+1)}$.
To see this, one can easily compute the exact expression for $\beta:=\dist(q,r)$ solving the equation
\[
h=(n-1)\,\omega_{n-1}\int_{0}^{\sqrt{2\beta-\beta^2}}\left(\sqrt{1-r^2}-1+\beta\right)r^{n-2}\,dx.
\]
Alternatively, one can simply notice that $\dist(q,s)\simeq\beta^{1/2}$ and the volume of $B_1(p)\setminus\Omega$ is comparable with that of the cylinder with base $\de\Omega\cap B_1(p)$ and height $\beta$ (in fact, the cylinder with half the height and half the radius of the base is contained in $B_1(p)\setminus\Omega$). Hence, $h\simeq \beta^{(n+1)/2}$ from which the conclusion. Clearly, $v=-\beta\,e_n$ fulfills the conclusion of the lemma.

If $\Omega$ is not flat, we need to restrict the size of the balls we consider choosing $r_0$ small enough to have $|A_{\de\Omega}|\leq \eps(n)\,r_0^{-1}$,
% \begin{equation}%\label{e.II}
% |A_{\de\Omega}|\leq \eps(n)\,r_0^{-1},
% \end{equation}
where $A_{\de\Omega}$ is the second fundamental form of $\de \Omega$ and $\eps(n)>0$ is a dimensional constant to be choosen momentarily.
Consider $r\leq r_0$ and $p$ as in the statement.
By a simple rescaling of the variable by a factor $r$ and a translation, we find $B_1(p')$ and new domain $\Omega'$ such that $|B_1(p')\setminus\Omega'|\leq h\leq h_0$ and
\begin{gather}\label{e.flat}
\de\Omega'\cap B_1(p')\subset \big\{(x,y)\in\R^{n-1}\times\R\,:\,-\eps(n)\,|x|^2\leq y\leq \eps(n)\,|x|^2\big\}.
% ,\\
% |B_1(p')\setminus\Omega'|\leq h\leq h_0.
\end{gather}
% 
% 
% \begin{itemize}
% \item[(a)] $\Omega'\cap B_1(p')\subset \{(x,y)\in\R^{n-1}\times\R\,:\,-\eps(n)\,|x|^2\leq y\leq \eps(n)\,|x|^2\}$;
% % \[
% % \Omega'\cap B_1(p')\subset \big\{(x,y)\in\R^{n-1}\times\R\,:\,-\eps(n)\,|x|^2\leq y\leq \eps(n)\,|x|^2\big\};
% % \]
% \item[(b)] $|B_1(p')\setminus\Omega'|\leq h\leq h_0$.
% \end{itemize}
Note that, by an analogous computation as above (now $\beta:=\dist(q',r')$), we have that
\begin{multline*}
(n-1)\,\omega_{n-1}\int_{0}^{\sqrt{2\beta-\beta^2}}\left(\sqrt{1-r^2}-1+\beta-\eps(n)\,r^2\right)r^{n-2}\,dx
\leq h\\
\leq
(n-1)\,\omega_{n-1}\int_{0}^{\sqrt{2\beta-\beta^2}}\left(\sqrt{1-r^2}-1+\beta+\eps(n)\,r^2\right)r^{n-2}\,dx.
\end{multline*}
One can easily compute (or argue by elementary geometric consideration as before) that $h\simeq\beta^{(n+1)/2}$.
Note that, setting as before $v':=-\beta\,e_n$, we have that $B_1(p'+v')\subset\Omega'$ because of \eqref{e.flat}.
Scaling back to $\Omega$, the conclusion hence follows.
\end{proof}

\subsection{Proof of Theorem~\ref{t.main-0}: part I}
Here we prove Theorem~\ref{t.main-0} (i), (ii) and (iii).

The proof of (i) follows from the same arguments in Theorem~\ref{t.periodic}. Indeed, thanks to the localization in Proposition~\ref{p.local}, we are now in the same position to discard the boundary effects. More precisely, by Proposition~\ref{p.quasi min} and Lemma~\ref{l.vol constr}, we know that the $E_m$ are uniform $\Lambda$-minimizers. Hence, by the regularity in Proposition~\ref{p.improved} the sets $E_m$ can be parametrized on a optimal isoperimetric ball $B_{E_m}^{opt}$ by a $C^{3,\alpha}$ regular function. Hence, we can derive for $E_m$ the improved perimeter estimate in Proposition~\ref{p.deficit} and use the optimal isoperimetric inequality to conclude \eqref{e.rate-0}.

For what concerns (ii), let $q\in\cH$ be a generic harmonic center and let $p_m^{opt}$ the center of the optimal ball for $E_m$, namely $B_{E_m}^{opt}=B_{r_m}(p_m^{opt})$. Comparing the energy of $E_m$ with that of $B_{r_m}(q)$ and using that, as shown in the proof of Theorem \ref{t.periodic}, it holds $|{E_m}\sdif B_{E_m}^{opt}|\lesssim\gamma\,r_m^{2n+2}$,
% \begin{equation*}
% |{E_m}\sdif B_{E_m}^{opt}|\lesssim\gamma\,r_m^{2n+2},
% \end{equation*}
we get:
\begin{align}\label{e.ball vs E-2}
\gamma\,r_m^{2n}\,g_{r_m}(p_m^{opt})-\gamma\,r_m^{2n} g_{r_m}(q)&=F_{\gamma,m}(B_{E_m}^{opt})-F_{\gamma,m}(B_{r_m}(q))
%\notag\\&
\leq F_{\gamma,m}(B_{E_m}^{opt})-F_{\gamma,m}(E_m)\notag\\
&\leq
% \Per(B_m)-\Per(E_m)+
\gamma\,NL(B_{E_m}^{opt})-\gamma\,NL(E_m)
%\notag\\&
\lesssim
\gamma\,\frac{|{E_m}\sdif B_{E_m}^{opt}|^2}{ r_m^{n-2}\,}+\gamma\,r_m^{n+1}\,|{E_m}\sdif B_{E_m}^{opt}|
\notag\\
&
\lesssim \gamma^2\,r_m^{3n+3}=\gamma\,\delta_0\, r_m^{3n}.
\end{align}
This implies by \eqref{e.g0} in Lemma \ref{l.balls} that 
\[
h(p_m^{opt})-h(q)=g_{r_m}(p_m^{opt})-g_{r_m}(q)+C\,r_m^2
% =\fint_{B_{r_m}}\fint_{B_{r_m}} \left(R(x+p_m,y+p_m)-R(x+q,y+q)\right)\,dx\,dy
\lesssim \delta_0\, r_m^{n}+r_m^2\lesssim r_m^2.
\]
Since the minimum points of $h$ are isolated, from this estimate it follows that $p_m^{opt}$ belongs to some neighborhood of the harmonic centers.

Finally, prove (iii) follows readly as in Theorem~\ref{t.periodic} comparing with the energy of $B_{r_m}(p_m^{opt})$. \qed

%%%%%%%%%%%%%%%%%%%%% Stability %%%%%%%%%%%%%%%%%%%%%%%%%%%%%%%%%%%%%
%%%%%%%%%%%%%%%%%%%%%%%%%%%%%%%%%%%%%%%%%%%%%%%%%%%%%%%%%%%%%%%%%%%%%
%%%%%%%%%%%%%%%%%%%%%%%%%%%%%%%%%%%%%%%%%%%%%%%%%%%%%%%%%%%%%%%%%%%%%
%%%%%%%%%%%%%%%%%%%%%%%%%%%%%%%%%%%%%%%%%%%%%%%%%%%%%%%%%%%%%%%%%%%%%
%%%%%%%%%%%%%%%%%%%%%%%%%%%%%%%%%%%%%%%%%%%%%%%%%%%%%%%%%%%%%%%%%%%%%
\section{Stability and exact solutions}\label{s.stable}
In this section we address the problem of the formation of exact spherical droplets, proving assertion (iv) in Theorem~\ref{t.main-0}.

\subsection{Non spherical domains: non existence of critical spherical droplets}\label{ss.general-dom}
In this section we show that if $\Omega\neq B_R$ is not itself a ball, the critical points of $F_{\gamma,m}$ cannot be exactly spherical.

\begin{proposition}\label{p.no ball}
Let $\Omega\subset\R^n$ be a $C^2$ bounded open set and assume that it is not a ball. Then, $\chi_{B_{r_m}(p)}$ with $B_{r_m}(p)\subset\Omega$ is not a critical point of $F_{\gamma,m}$.
\end{proposition}

\begin{proof}
The proof is a simple consequence of a unique continuation argument. Indeed, we show that if $\chi_{B_{r_m}(p)}$ satisfies the Euler--Lagrange equation \eqref{e.first var}, \eqref{e.potential}, namely
\begin{equation}\label{e.first var m}
\begin{cases}
H_{\de B_{r_m}(p)}+4\gamma\, v_m=\lambda_m,&\\
-\Delta v_m=\chi_{B_{r_m}(p)}-m & \text{in } \ \Omega,\\
\nabla v_m\cdot \nu =0 & \text{on }\de\Omega,\\
\int_\Omega v_m=0,&
\end{cases}
\end{equation}
then $v_m$ is a radially symmetric function with respect to $p$, and hence $\Omega$ must be a ball.
Assume without loss of generality that $p=0$ and \eqref{e.first var m} holds, and consider the case $n\geq 3$ (the two dimensional case is analogous). Since $H_{\de B_{r_m}}\equiv(n-1)/r_m$, it follows from the first equation in \eqref{e.first var m} that $v_m\vert_{\de B_{r_m}}\equiv c_m\in\R$. Thus, from the uniqueness for the Dirichlet problem for the Laplacian, we infer that $v_m\vert_{B_{r_m}}$ is radially symmetric and:
\begin{equation}\label{e.vball}
v_m(x)=\frac{(1-m)\,(|x|^2-r_m^2)}{2n}+c_m,\quad \text{for } |x|\leq r_m.
\end{equation}
Moreover, in $\Omega\setminus B_{r_m}$, $v_m$ solves the boundary value problem:
\begin{equation}\label{e.bd pb}
\begin{cases}
\Delta v_m=-m &\text{in } \Omega\setminus B_{r_m},\\
v_m=c_m,\quad \nabla v_m\cdot \nu_{\de B_{r_m}}=\frac{(1-m)\,r_m}{n} &\text{on } \de  B_{r_m}.
% \\
% %\nabla v\cdot \eta=c_1 &\text{on } \de  B_{r_m},\\
% \nabla v\cdot \eta=0 &\text{on }  \de\Omega.
\end{cases}
\end{equation}
Note that also \eqref{e.bd pb} has a unique solution. Indeed, given $v_1, v_2$ solving \eqref{e.bd pb}, $w=v_1-v_2$ solves
\begin{equation}\label{e.bd pb2}
\begin{cases}
\Delta w=0 &\text{in } \Omega\setminus B_{r_m},\\
w=\nabla w\cdot \eta=0 &\text{on } \de  B_{r_m},
% \\
% %\nabla v\cdot \eta=c_1 &\text{on } \de  B_{r_m},\\
% \nabla v\cdot \eta=0 &\text{on }  \de\Omega.
\end{cases}
\end{equation}
which is extended to a harmonic function in $\Omega$ setting $w\equiv0$ in $B_{r_m}$, thus implying $w\equiv0$ in $\Omega\setminus B_{r_m}$.
By a direct computation, the solution of \eqref{e.bd pb} is given by
\[
v_m(x):=-\frac{m\,(|x|^2-r_m^2)}{2n}+c_m
+\frac{r_m^2}{n\,(n-2)}-\frac{r_m^n}{n\,(n-2)\,|x|^{n-2}}.
\]
Therefore, since $\nabla v_m\cdot\nu\equiv 0$ on $\de \Omega$, it follows by the radial symmetry of $v_m$ that $\Omega$ is a ball, which contradicts the hypothesis. This gives the desired contradiction and concludes the proof.
\end{proof}

\begin{remark}
 In particular, in the case of periodic boundary conditions the exact sphere is never an equilibrium configuration.
\end{remark}

\subsection{Ball domains: uniqueness of a spherical droplet minimizer}\label{ss.ball-drop}
In this section we take $\Omega=B_R$. In this case we show that the ball $B_{r_m}$ is the only minimizer of $F_{\gamma,m}$ in the regime of small mass, thus completing the proof of Theorem~\ref{t.main-0}. In order to address this problem, here we need to introduce a new ingredient: the stability analysis of the droplet configuration. In particular, we will show that the spherical droplet $B_{r_m}$ is strictly stable which will turn to imply that it is the unique minimizer of $F_{\gamma,m}$.

\begin{proposition}\label{e.stable}
Assume $\Omega=B_R\subset\R^n$, for some $R>0$. There exists $\delta_0>0$ such that, if $\gamma\,r_m^3\leq \delta_0$ in the case $n\geq 3$ and if $\gamma\,r_m^3|\log r_m|\leq \delta_0$ in the case $n=2$, then $B_{r_m}$ is the unique minimizer of $F_{\gamma,m}$.
\end{proposition}

\begin{proof}
The proof of the proposition is made in three steps.

\medskip

\textsc{Step 1.} \textit{The minimizers $E_m$ can be parametrized on $B_{r_m}$ for $\delta_0$ small enough}.

To see this, we recall that in the case $\Omega=B_R$, due to the spherical symmetry, the the origin is the only minimum point of the Robin function. Moreover, $ D^2h (0)\gtrsim \Id$.
To check this, one can either use thethe explict formula for $h$ (see, e.g., \cite[Chapter IV 5]{Leis} in the case $n=3$, similar formulas hold in every dimension):
\begin{equation*}%\label{e.h B}
 h(x) = 
% \begin{cases}
%  \frac{1}{2\,\pi}\,\log\left(\frac{R^2}{R^2-|x|^2}\right) + \frac{|x|^2}{2\,\pi\,R^2}+h(0), & \text{if } \ n=2,\\
 \frac{R\,|x|^{n-3}}{(R^2-|x|^2)^{n-2}}+\frac{1}{R^{n-2}}\log\left(\frac{R^2}{R^2-|x|^2}\right) + \frac{|x|^2}{2\,n\,\omega_n\,R^n} +h(0), \quad \text{if } \ n\geq 3;
% \end{cases}
\end{equation*}
or, one can simply notice that $R(x,0)=\frac{|x|^2}{2\,n\,\omega_n R^n}$, so that $D^2h(0)=D_x^2R(0,0)=\frac{\Id}{n\omega_nR^n}$. From the definition of $g_r$ in \eqref{e.gr} and the radial symmetry of $h$, it is readly verified that $g_r$ also has minimum in the origin and this minimum is not degenerate as well. From \eqref{e.ball vs E-2}, we can hence conclude that $|p_m^{opt}|^2\lesssim \delta_0\,r_m^n$. Note that, in any dimension $n$, this implies that
\begin{equation}\label{e.p opt}
|p_m^{opt}|\lesssim \delta_0^{1/2}\, r_m. 
\end{equation}

This actually leads straightforwardly to the claim. Indeed, for $\delta_0$ small enough, there exists $s<1$ such that, for every point $x\in \de B_{r_m}$, $B_{sr_m}(x)\cap B_{r_m}^{opt}$ is a graph over $\de B_{r_m}$ with small Lipschitz constant. Since by (i) of Theorem~\ref{t.main-0} the sets $E_m$ are parametrized on $\de B_{r_m}^{opt}$ with a graph of small $C^1$-norm, this implies that in turns also $\de E_m$ is a graph on $\de B_{r_m}$. Moreover, the $C^{3,\alpha}$ regularily is clearly preserved for this new parametrization.

\medskip

\textsc{Step 2.} \textit{We show now that for $\delta_0$ small enough, the ball $B_{r_m}$ is strictly stable.}

Let us recall the second variation for $F_{\gamma,m}$. Let $E$ be a stationary point and consider vector fields $X\in C^1_c(\Omega,\R^n)$ such that
\begin{equation}\label{e.admissible}
\int_{\de E} X\cdot\nu_E\,d\cH^{n-1}=0.
\end{equation}
Following \cite[Lemma 2.4]{BdC}, for every such field, there exists $F:\Omega\times(-\eps,\eps)\to\Omega$ such that:
\begin{itemize}
\item[(a)] $F(x,0)=x$ for all $x\in\Omega$, $F(x,t)=x$ for all $x\in\de \Omega$ and $t\in(-\eps,\eps)$;
\item[(b)] $E_t:=F(E,t)$ satisfies $|E_t|=|E|$ for every $t\in(-\eps,\eps)$;
\item[(c)] $\frac{\de F(x,t)}{\de t}\vert_{t=0}=X(x)$ for every $x\in \de E$.
\end{itemize}
The stability operator for deformations as above is given by (see \cite{Mu} in the case $n=2,3$ and \cite{AcFuMo, ChSt} in the general case)
\begin{align}\label{e.stability op}
F_{\gamma,m}''(E)[X]:=&{}\Per''(E)[X]+\gamma\,\NL''(E)[X]\notag\\
=&{}\int_{\de E}\left(|\nabla_{\de E}(X\cdot\nu_E)|^2-|A|^2\,(X\cdot\nu_E)^2\right)\,d\cH^{n-1}+\notag\\
&+8\,\gamma\,\int_{\de E}\int_{\de E}G(x,y)\,(X(x)\cdot\nu_E)\,(X(y)\cdot\nu_E)\,d\cH^{n-1}(x)\,d\cH^{n-1}(y)\,+\notag\\
&+4\,\gamma\,\int_{\de E}\nabla v\cdot\nu_E\,(X\cdot\nu_E)^2\,d\cH^{n-1},
\end{align}
where $|A|$ is the lenght of the second fundamental form of $\de E$ and $v$ solves \eqref{e.potential}.

In order to prove the strict stability of $B_{r_m}$ we need to compute \eqref{e.stability op} on $E=B_{r_m}$ and show the existence of a constant $c_0(n,m,\gamma)> 0$ such that
\begin{equation}\label{e.stability}
F_{\gamma,m}''(B_{r_m})[X]\geq c_0\,\|X\cdot\nu_{B_{r_m}}\|_{L^2(\de B_{r_m})}^2,
\end{equation}
for every $X$ as in \eqref{e.admissible}. We start by noticing that the following inequality holds true: there exists a constant $c_1(n)>0$ such that
\begin{equation}\label{e.P''}
\Per''(B_{r_m})[X]\geq \frac{c_1(n)}{r_m^2}\,\,\|X\cdot\nu_{B_{r_m}}\|_{L^2(\de B_{r_m})}^2,\quad\forall \ X\vert_{\de B_{r_m}}\neq \textup{const}.
\end{equation}
%for all nonconstant $X$ in \eqref{e.admissible}.
This is a simple consequence of the properties of the spectrum of the Laplace--Beltrami operator on the sphere $-\Delta_{\de B_{r_m}}$. Indeed, by the decomposition in spherical harmonics  (see, e.g., \cite[Chapter 5]{HFT}), let write $X\cdot\nu_{B_{r_m}}=\sum_{i\in\N^+} f_i$, with $f_i$ spherical harmonic in the $i^{\textup{th}}$ eigenspace of the Laplace--Beltrami operator, i.e. $-\Delta_{\de B_{r_m}} f_i = \lambda_i\,f_i$ and $0<\lambda_1<\lambda_2<\ldots$ is the discrete spectrum of the operator. Note that $i$ ranges over $\N^+$ because the harmonic functions on the sphere (i.e., the constant functions) are ruled out by the average condition \eqref{e.admissible}. If $\Per''(B_{r_m})[X]=0$, i.e.
\[
\int_{\de B_{r_m}}\left(|\nabla_{\de B_{r_m}}(X\cdot\nu_{B_{r_m}})|^2-\frac{n-1}{r_m^2}\,(X\cdot\nu_{B_{r_m}})^2\right)\,d\cH^{n-1}=0,
\]
then, by a simple integration by parts and using the orthogonality of the eignespaces, it follows that
\[
 \sum_{i\in\N^+} \left(\lambda_i-\frac{n-1}{r_m^2}\right)\,\|f_i\|_{L^2(\de B_{r_m})}=0
\]
As it is well-know, $\lambda_1=\frac{n-1}{r_m^2}$. Hence, $\Per''(B_{r_m})[X]=0$ if and only if $X\cdot \nu_{B_{r_m}}=f_1$ is in the first eigenspace, that is $X\vert_{\de B_{r_m}}$ is constant. Then, the existence of a constant $c_1$ fulfilling \eqref{e.P''} follows by the discreteness of the spectrum $-\Delta_{B_{r_m}}$ and the obvious scaling property in $r_m$.

We now recall that the second term in $F''_{\gamma,m}$ is positive (see \cite{ChSt}); while the last is equal to (for $E=B_{r_m}$ here $v=v_m$ is given in \eqref{e.vball})
\[
4\,\gamma\,\int_{\de B_{r_m}}\nabla v_m\cdot\nu_{B_{r_m}}\,(X\cdot\nu_{B_{r_m}})^2\,d\cH^{n-1}=-\frac{4\,\gamma\,(1-m)\,r_m}{n}\,\|X\cdot\nu_{B_{r_m}}\|_{L^2(\de B_{r_m})}^2.
\]
Moreover, for a constant vector field $X=\zeta\in\R^n$, we can compute explicitely $\NL''(B_{r_m})[\zeta]$ in the following way:
\begin{align}\label{e.NL'' v}
\NL''(B_{r_m})[\zeta]=\frac{d^2 NL(B_{r_m}(t\zeta))}{d\,t^2}\Big\vert_{t=0}&=\frac{d^2}{d\,t^2}\int_{B_{r_m}}\int_{B_{r_m}} \left(\Gamma(|x-y|)+R(x+t\zeta,y+t\zeta)\right)\,dx\,dy\Big\vert_{t=0}\notag\\
&=\int_{B_{r_m}}\int_{B_{r_m}}(D^2R(x,y)(\zeta,\zeta),(\zeta,\zeta))\,dx\,dy\gtrsim |\zeta|^2\notag\\
&\gtrsim \|\zeta\cdot\nu\|_{L^2(\de B_{r_m})}^2,
\end{align}
because we have already noticed that for $\Omega=B_R$ the regular part of the Green function has in the origin the unique nondegenerate minimum, i.e.~$D^2R(0,0)=\mu \,\Id$ for some $\mu>0$.

The proof of \eqref{e.stability} is now clear. For a nonconstant $X$, \eqref{e.stability} follows if $\gamma\,r_m^3$ is small enough with respect to $c_1(n)$. For a constant vector field $X$ \eqref{e.stability} is inferred from \eqref{e.NL'' v}.

\medskip

\textsc{Step 3.} \textit{$B_{r_m}$ is the unique minimizer}.
The conclusion follows from the fact that the minimizers $E_m$ are $C^2$ close to a strictly stable configuration, namely $B_{r_m}$, thus implying that actually $E_m$ coincide with $B_{r_m}$. The proof of this fact, well-known for the area functional, can be achieved by a carefull construction of a flow interpolating $\de E_m$ and $\de B_{r_m}$. Such computations appeared in \cite[Theorem~3.9]{AcFuMo}. In particular, to reduce to this case, let $\psi_m$ be the parametrization of $\de E_m/r_m$ on $\de B_{1}$, i.e.
\[
 E_m = \big\{r_m\,x(1+\psi_m(x)) \, : \, x\in \de B_1\big\}.
\]
By Proposition~\ref{p.improved} and \eqref{e.p opt}, for every $\eta>0$ we can choose $\delta_0$ small enough to have $\|\psi_m\|_{C^{3,\alpha}}\leq \eta$. We are, hence, a small perturbation of the fixed stable configuration $B_1$ and \cite[Theorem~3.9]{AcFuMo} applies.
\end{proof}
\begin{remark}
The proof of the previous result becomes trivial if the minimizer $E_m$ is such that $B_{E_m}^{opt}$ is centered at the origin,that is $B_{E_m}^{opt}=B_{r_m}$. In this case one can indeed drop the quadratic term in \eqref{e.deficit periodic} and, by the quantitative isoperimetric inequality, get
\begin{equation*}
|B_{r_m}\sdif E_m|^2\lesssim\gamma\, r_m^3 |B_{r_m}\sdif E_m|^2,
\end{equation*}
which turns to imply $E_m=B_{r_m}$.
\end{remark}

%%%%%%%%%%%%%%%%%%%%% Appendix %%%%%%%%%%%%%%%%%%%%%%%%%%%%%%%%%%%%%%
%%%%%%%%%%%%%%%%%%%%%%%%%%%%%%%%%%%%%%%%%%%%%%%%%%%%%%%%%%%%%%%%%%%%%
%%%%%%%%%%%%%%%%%%%%%%%%%%%%%%%%%%%%%%%%%%%%%%%%%%%%%%%%%%%%%%%%%%%%%
%%%%%%%%%%%%%%%%%%%%%%%%%%%%%%%%%%%%%%%%%%%%%%%%%%%%%%%%%%%%%%%%%%%%%
%%%%%%%%%%%%%%%%%%%%%%%%%%%%%%%%%%%%%%%%%%%%%%%%%%%%%%%%%%%%%%%%%%%%%

\appendix
\section{On the behavior of the function $R$ in a neighborhood of $\de\Omega$}\label{a.Robin}
In this section we prove the estimates \eqref{e.Robin} and \eqref{e.robin0} on the regular part $R$ of the Green function in \eqref{e.G}. The function $R$ solves
\[
 \begin{cases}
  \Delta R_x = \frac{1}{|\Omega|} & \text{in }\;\Omega,\\
  \nabla R_x \cdot \nu = \nabla \Gamma_x \cdot \nu & \text{on }\;\de \Omega,\\
  \int_\Omega R_x = \int_\Omega \Gamma_x. &
 \end{cases}
\]

We introduce the following notation. Since $\Omega$ is assumed to have $C^2$ regular boundary, for $x$ in a sufficiently small tubolar neighborhood of $\de\Omega$, there exists a unique point $x_0 \in \de \Omega$ such that $\dist(x,\de\Omega)=|x-x_0|$. Hence, can hence consider $x^*\in\R^n\setminus\Omega$ such that $x^*-x_0 = x_0-x$ and set $S_x := R_x + \Gamma_{x^*}$. $S_x$ is also characterized by the following boundary value problem:
\begin{equation}\label{e.Sx}
 \begin{cases}
  \Delta S_x = \frac{1}{|\Omega|} & \text{in }\;\Omega,\\
  \nabla S_x \cdot \nu = (\nabla \Gamma_x + \nabla \Gamma_{x^*}) \cdot \nu & \text{on }\;\de \Omega,\\
  \int_\Omega S_x = \int_\Omega (\Gamma_{x}+\Gamma_{x^*}). &
 \end{cases}
\end{equation}

The main idea being the estimates are illustrated in the following simple case. Assume that $0\in\de\Omega$ and $B_2\cap\Omega=\{x\in B_2 \,:\, x_n<0 \}$. Then, by an elementary computation, for every $x\in B_1$,
\[
 \begin{cases}
  (\nabla \Gamma_x + \nabla \Gamma_{x^*}) \cdot \nu = 0 & \text{on }\; \de\Omega \cap B_2,\\
  \left| (\nabla \Gamma_x + \nabla \Gamma_{x^*}) \cdot \nu \right| \lesssim 1 & \text{on }\;\de\Omega \setminus B_2.
 \end{cases}
\]
Therefore, it follows from \eqref{e.Sx} that $|S_x|\leq C$. This in turns implies \eqref{e.Robin}: namely, there exists $r_0>0$ such that, for $r\leq r_0$ and $x,y\in
\Omega \cap B_1$ with $r<-x_n<2\,r$ and $|x-y|\leq r$,
\[
 |R_x(y)|\simeq |\Gamma_{x^*}(y)|\simeq |\Gamma(r)|.
\]
Moreover, since for $|\Gamma_{x_*}|\lesssim |\Gamma_x|+1$ for every $x\in B_1$, \eqref{e.robin0} follows straightforwardly as well. 

The general case of a $C^2$ bounded domain $\Omega$ can be deduced by a perturbation of the argument above.
Let $r_0>0$ be such that, for every $x_0\in\de\Omega$, $B_{2\,r_0}(x_0)\cap \de\Omega$ can be written as the graph of a function: namely, up to an affine change of coordinates, we
may assume that $x_0=0$ and
\[
 B_{2\,r_0} \cap \Omega = \{(z',t) \;:\; t\leq u(z') \},
\]
for a given $u:B_{2\,r_0}^{n-1}\subset \R^{n-1}\to \R$ in $C^2 (B_{2\,r_0}^{n-1})$ with $u(0)=|\nabla u(0)|=0$. In particular, we have $x=(0,-d)$ with $d=\dist(x,\de\Omega)$. Set $D := B_{2\,r_0}^{n-1} \times [0,1]$ and consider the function
$g:D \to \R$ given by
\[
 g(z',t) := \Big(\nabla \Gamma_x \big( (z',t\,u(z')) \big) + \nabla \Gamma_{x^*} \big( (z',t\,u(z')) \big)\Big) \cdot \frac{(-t\,\nabla u(z'),1)}{\sqrt{1+t^2\,|\nabla u(z')|^2}}.
\]
By definition, $g(z',1)=\nabla S_x \cdot \nu\vert_{\de \Omega}$ and $g(z',0)=0$. Writing $z_t=(z',t\,u(z'))$, it holds
\begin{align*}
 \de_t \, g(z',t) = {} & u(z')\,\Big( \frac{\de}{\de x_n} \nabla \Gamma_x(z_t) + \frac{\de}{\de x_n} \nabla \Gamma_{x^*}(z_t) \Big) \cdot \nu\vert_{\de \Omega_t} +\\
& - \Big( \nabla
\Gamma_x (z_t)+ \nabla \Gamma_{x^*} (z_t) \Big) \cdot \frac{\nabla u (z')}{\sqrt{1+t^2\,|\nabla u(z')|}}+\\
& - \frac{t\,|\nabla u(z')|^2}{1+t^2|\nabla u(z')|^2}\, \Big( \nabla \Gamma_x+ \nabla \Gamma_{x^*} \Big) \cdot \nu\vert_{\Omega_t}(z_t).
\end{align*}
Since $|u(z')|\leq C\, |z'|^2$ and $|\nabla u(z')|\leq C\, |z'|$, where $C>0$ depends only on $\|u\|_{C^2}$, one infers that
\begin{align}\label{e.deriv}
 |\de_t \, g(z',t)| & \lesssim \Big( |D^2 \Gamma_x (z_t)|+|D^2 \Gamma_{x^*}(z_t)| \Big)\,|z'|^2 + \Big( |\nabla \Gamma_x|+|\nabla \Gamma_{x^*}| \Big)\,|z'| \notag\\
& \lesssim \frac{|z'|^2}{|x-z_t|^n} + \frac{|z'|}{|x-z_t|^{n-1}}.
\end{align}
Note that, for $|z'|\leq r_0$ small enough,
\begin{align*}
 |x-z_t|^2&=|d+tu(z')|^2+|z'|^2\geq \frac{d^2}{2}-|u(z')|^2+|z'|^2\geq \frac{d^2}{2}-C\,|z'|^4+|z'|^2\\
&\geq \frac{d^2}{2} + \frac{|z'|^2}{2},
\end{align*}
from which we infer
\begin{equation}\label{e.universal}
  |\de_t \, g(z',t)| \lesssim \frac{|z'|^2}{|x-z_t|^n} + \frac{|z'|}{|x-z_t|^{n-1}}\lesssim \frac{|z'|}{(d^2+|z'|^2)^{\frac{n-1}{2}}} =: f(z').
\end{equation}
It is simple to see that $f\in L^p(B_{2r_0}^{n-1})$ for every $p\in [1,\infty]$ and
\begin{align}\label{e.pnorm}
 \int_{B_{2r_0}^{n-1}}f(z')^p\,dz' & = \int_{0}^{2r_0} \frac{s^p}{(d^2+s^2)^{\frac{p(n-1)}{2}}}\,s^{n-2}\,ds\notag\\
&= d^{-p\,(n-2)+n-1}\int_{0}^{\frac{2r_0}{d}} \frac{t^{p+n-2}}{(1+t^2)^{\frac{p(n-1)}{2}}}\,dt\notag\\
& \lesssim
\begin{cases}
  2r_0 & \text{if }\; n=2,\\
 d^{-p\,(n-2)+n-1}\int_{0}^{\infty} \frac{t^{p+n-2}}{(1+t^2)^{\frac{p(n-1)}{2}}}\,dt \lesssim d^{-p\,(n-2)+n-1} & \text{if } \, n\geq3,
\end{cases}
\end{align}
Since, for $\dist(x,\de\Omega)\leq d_0$, and every $z\in \de\Omega\cap B_{2r_0}$,
\[
 |\nabla S_x (z) \cdot \nu\vert_{\de \Omega\cap B_{2r_0}}| = |g(z',1) - g(z',0)| \leq \int_0^1 |\de_t g(z',s)|\,ds,
\]
we deduce the following bound on the $L^p$ norm of $\nabla S_x \cdot \nu$:
\begin{align*}
\|\nabla S_x \cdot \nu\|^p_{L^p(\de\Omega\cap B_{2r_0})}\lesssim\, &\int_{\de\Omega\cap B_{2r_0}} \left(\int_0^1 |\de_t g(z',s)|\,ds\right)^p\,dz'+ \int_{\de\Omega\setminus B_{2r_0}}|\nabla S_x \cdot \nu|^p \,dz'\\
\lesssim&\int_{\de\Omega\cap B_{2r_0}} \int_0^1 |\de_t g(z',s)|^p\,ds\,dz'+ C\\
\stackrel{\eqref{e.pnorm}}{\lesssim}&
\begin{cases}
  2r_0 & \text{if }\; n=2,\\
 d^{-p\,(n-2)+n-1}\int_{0}^{\infty} \frac{t^{p+n-2}}{(1+t^2)^{\frac{p(n-1)}{2}}}\,dt \lesssim d^{-p\,(n-2)+n-1} & \text{if } \, n\geq3.
\end{cases}
\end{align*}
Setting $\beta=(n-1)/p>0$, as a consequence, by $L^p$-regularity theory for \eqref{e.Sx}, we get $\|S_x\|_{W^{1,p}}\lesssim d_0^\beta\dist(x,\de\Omega)^{2-n}$. 
By the arbitrariness of $p$ and thanks to Sobolev embedding we finally get that 
\begin{equation}\label{e.ebasta}
|S_x|\lesssim d_0^\beta\dist(x,\de\Omega)^{2-n}.
\end{equation}
The proofs of \eqref{e.robin0} and \eqref{e.robin} now follows straightforwardly. 

\subsection{Proof of \eqref{e.robin0}}
Fix $r_0\leq d_0$ as above. Then, if $x$ belongs to some $\Omega_{r_0}$, then
\begin{align*}
 |G(x,y)| &\leq |\Gamma(x,y)| + |R(x,y)| \\
 & \leq |\Gamma(x,y)| + |\Gamma(x^*,y)| + |S_x| \\
 & \lesssim |\Gamma(x,y)| + 1,
\end{align*}
where we used that $\dist(x,\de\Omega)^{2-n}\lesssim |\Gamma(x,y)|$ for every $y\in\Omega$ if $n\geq3$, and $\dist(x,\de\Omega)\lesssim 1$ in $n=2$.

\subsection{Proof of \eqref{e.robin}}
Note that, for $r_0$ small enough, whenever $r<r_0/2$, $r\leq \dist(x,\de\Omega)\leq 2\,r$ and $|y-x|\leq r$, then $|\Gamma_{x^*}(y)|\simeq |\Gamma_{x}(y)|$.
Then, by \eqref{e.ebasta} we may assume that $d_0$ is sufficiently small that, for $r_0\leq d_0$ and $x.y$ as above, it holds 
\[
 |R_x(y)|\geq |\Gamma_{x^*}(y)|-|S_x(y)|\gtrsim |\Gamma_{x}(y)| - d_0^\beta\,\dist(x,\de\Omega)^{2-n}\gtrsim |\Gamma_{x}(y)|,
\]
and
\[
 |R_x(y)|\leq |\Gamma_{x^*}(y)|+|S_x(y)|\lesssim |\Gamma_{x}(y)| + d_0^\beta\, \dist(x,\de\Omega)^{2-n}\lesssim |\Gamma_{x}(y)|.
\]

\subsection*{Acknowledgements}
The research of the first author was partially supported by the European Research Council under FP7,
Advanced Grant n. 226234 ``Analytic Techniques for Geometric and Functional Inequalities" and by
the Deutsche Forschungsgemeinschaft through the Sonderforschungsbereich 611 during his stay at the Institute for Applied Mathematics of the University of Bonn.

%%%%%%%%%%%%%%%%%%%%% Bibliography %%%%%%%%%%%%%%%%%%%%%%%%%%%%%%%%%%
%%%%%%%%%%%%%%%%%%%%%%%%%%%%%%%%%%%%%%%%%%%%%%%%%%%%%%%%%%%%%%%%%%%%%
%%%%%%%%%%%%%%%%%%%%%%%%%%%%%%%%%%%%%%%%%%%%%%%%%%%%%%%%%%%%%%%%%%%%%
%%%%%%%%%%%%%%%%%%%%%%%%%%%%%%%%%%%%%%%%%%%%%%%%%%%%%%%%%%%%%%%%%%%%%
%%%%%%%%%%%%%%%%%%%%%%%%%%%%%%%%%%%%%%%%%%%%%%%%%%%%%%%%%%%%%%%%%%%%%
\bibliographystyle{plain}
\bibliography{bib-CicSpa_2011}

\end{document}